\documentclass[twoside]{article}

\usepackage{a4wide, amsmath,  amssymb, amscd, mathptmx, amsthm}
\usepackage{amsfonts}
\usepackage{stmaryrd}
\usepackage{bbm}

\usepackage{color}

\usepackage{pstricks}
\usepackage[all]{xy}\CompileMatrices\SelectTips{cm}{12}

\theoremstyle{plain}
\newtheorem{Thm}{\sc Theorem}[section]
\newtheorem{theorem}[Thm]{\sc Theorem}
\newtheorem{corollary}[Thm]{\sc Corollary}
\newtheorem*{corollary*}{\sc Corollary}
\newtheorem{proposition}[Thm]{\sc Proposition}
\newtheorem*{proposition*}{\sc Proposition}
\newtheorem{lemma}[Thm]{\sc Lemma}
\newtheorem{conjecture}[Thm]{\sc Conjecture}
\newtheorem{question}[Thm]{\sc Questions}
\newtheorem{definition}[Thm]{\sc Definition}

\theoremstyle{remark}
\newtheorem{remark}[Thm]{Remark}
\newtheorem{example}[Thm]{Example}
\newtheorem*{example*}{Example}
\newtheorem*{remark*}{Remark}

\newcommand{\cA}{{\mathcal A}}

\newcommand{\cC}{{\mathcal C}}
\newcommand{\cD}{{\mathcal D}}

\newcommand{\cH}{{\mathcal H}}

\newcommand{\cM}{{\mathcal M}}

\newcommand{\cO}{{\mathcal O}}

\newcommand{\cV}{{\mathcal V}}

\newcommand{\cX}{{\mathcal X}}

\renewcommand{\AA}{{\mathbb A}}

\newcommand{\CC}{{\mathbb C}}

\newcommand{\GG}{{\mathbb G}}

\newcommand{\NN}{{\mathbb N}}

\newcommand{\PP}{{\mathbb P}}
\newcommand{\QQ}{{\mathbb Q}}

\newcommand{\VV}{{\mathbb V}}
\newcommand{\WW}{{\mathbb W}}
\newcommand{\ZZ}{{\mathbb Z}}

\newcommand{\CH}{{\mathop{{\rm CH}}}}

\newcommand{\Gr}{\mathop{\rm Gr}}

\newcommand{\Alb}{\mathop{\rm Alb \, }}

\newcommand{\Tr}{\mathop{\rm Tr}}

\newcommand{\id}{{\mathop{\rm id}}}

\newcommand{\Spec}{\mathop{\rm Spec \, }}

\newcommand{\Vect}{\mathop{{\rm Vect}}}

\newcommand{\alg}{\mathop{{\rm alg}}}

\newcommand{\rs}{\mathop{{\rm rs}}}

\newcommand{\an}{{\mathop{\rm an }}}
\newcommand{\ch}{{\mathop{\rm ch \, }}}
\newcommand{\chr}{\mathop{\rm char}}

\newcommand{\Hom}{{\mathop{{\rm Hom}}}}

\newcommand{\orb}{\mathop{{\rm orb}}}

\newcommand{\GL}{\mathop{\rm GL}}
\newcommand{\PGL}{\mathop{\rm PGL}}
\newcommand{\SL}{\mathop{\rm SL}}
\newcommand{\Sym}{{\mathop{\rm Sym}}}

\newcommand{\Pic}{{\mathop{\rm Pic\, }}}
\newcommand{\rk}{\mathop{{\rm rk}}}

\newcommand{\MIC}{{\mathop{{\rm MIC}}}}

\begin{document}

\markboth{\rm }{\rm  }

\title{On algebraic Chern classes of flat vector bundles }
\author{Adrian Langer}

\date{\today}

\maketitle

{\sc Address:}\\
Institute of Mathematics, University of Warsaw,
ul.\ Banacha 2, 02-097 Warszawa, Poland\\
e-mail: {\tt alan@mimuw.edu.pl}

\medskip

\begin{abstract} 	
We show that under some assumptions on the monodromy group some combinations of higher Chern classes of flat vector bundles are torsion in the Chow group. Similar results hold for flat vector bundles that deform to such flat vector bundles (also in case of quasi-projective varieties).
	The results are motivated by Bloch's conjecture on Chern classes of flat vector bundles on smooth complex projective varities but in some cases they give a more precise information. 
	
	We also study Higgs version of Bloch's conjecture and analogous problems in the positive characteristic case. 
\end{abstract}

\section*{Introduction}

Let $X$ be a smooth complex projective variety. In \cite{Bl2} Bloch conjectured that Chern classes of flat vector bundles on $X$ live in the lowest stratum of the (conjectural) Bloch--Beilinson filtration (see Conjecture \ref{Bloch's-conjecture}
for a more precise formulation).  A variant of this conjecture was proven by A. Reznikov in \cite{Re}: the Chern classes of a flat vector bundle on a smooth projective variety defined over $\CC$ vanish in Deligne's  cohomology $H^{2i}_{\cD} (X, \QQ(i))$ for $i\ge 2$. On the other hand, the Bloch--Beilinson conjecture says that if $X$ is defined over a number field then the cycle class maps  $\CH ^i (X ) _{\QQ}\to H^{2i}_{\cD} (X, \QQ (i))$ are injective for all $i\ge 1$. In particular, together with Reznikov's results this conjecture implies that if $X$ is defined over a number field then these Chern classes vanish in the Chow groups $\CH ^i (X) _{\QQ}$ for $i\ge 2$. One of the main results of this paper is the following result on Chern classes of families of flat vector bundles:

\begin{theorem}\label{main}
	Let $X$ be a smooth complex projective variety and let $S$ be an irreducible complex variety.
	Let $(V, \nabla)$ be a  rank $2$  vector bundle with a relative integrable connection on $X\times S /S$.
Assume that for  some $s_0\in S (\CC)$	the geometric monodromy group of $(V_{s_0}, \nabla _{s_0})$ contains $\SL(2, \CC)$. If the Bloch--Beilinson conjecture holds for certain explicit 
Shimura type varieties (see Conjecture \ref{Shimura-conj}) then there exists a positive integer $N$ such that for all $s\in S(\CC)$ we have 	
$$N\cdot  (4c_2(V_s)-c_1^2(V_s))=0 $$ 
	in $\CH ^2 (X)$.	
\end{theorem}

In fact, under some additional condition on non-rigidity we prove an unconditional result in an arbitrary rank.
An informal statement is as follows (see Theorem \ref{rk-r-families-torsion}).

\begin{theorem}\label{main2}
	Let $X$ be a smooth complex projective variety and let $S$ be an  irreducible complex variety.
	Let $(V, \nabla)$ be a  rank $r$ vector bundle with a relative integrable connection on $X\times S/ S$. 
	Assume that $(V_{\bar \eta}, \nabla_{\bar \eta})$, where $\bar \eta$ is the generic geometric point of $S$, is not projectively rigid and for some $s_0\in S (\CC)$ the geometric monodromy group $M_{\rm geom} (V_{s_0}, \nabla _{s_0})$ contains $\SL(n, \CC)$. Then for all $s\in S$  the characteristic classes of the $\PGL (r)$-bundles associated	to $V_s$  are torsion in $\CH ^r (X\otimes _{\CC} k(s))$ and the 
	order of torsion is uniformly bounded.
\end{theorem}

The above theorem allows us to reduce Theorem \ref{main} to  the rigid case that follows from the Bloch--Beilinson conjecture. We also prove a generalization of Theorem \ref{main} to families of flat vector bundles on quasi-projective varieties (see Theorem \ref{rk2-families-torsion}).

Note that even if $S$ is a point and $(V, \nabla)$ is a line bundle with an integrable connection, the first Chern class of $V$ need not be torsion (e.g., any non-torsion $\CC$-point of the Jacobian of a smooth complex projective curve gives rise to a unitary flat line bundle with non-torsion first Chern class).
Twisting a rank $2$ flat vector bundle by an appropriate flat line bundle shows that neither $c_2(V)$ nor $c_1^2(V)$ need to be torsion in general. It is more difficult to show examples when this fails for the second Chern class if the first Chern class is trivial but one can still find such examples (see Example \ref{counterexample}).  Theorem 
\ref{main2} implies that such examples have ``small'' geometric monodromy group and they do not live on the same irreducible components of the de Rham moduli space as flat vector bundles with ``large'' geometric monodromy.

\medskip

Bloch's conjecture implies also that Chern classes of a flat vector bundle on a smooth complex projective variety should be torsion in the ring $B^* (X)$ of cycles modulo algebraic equivalence. 

Let us recall that on a smooth complex projective variety C. Simpson in \cite{Si0}, motivated by earlier Hitchin's results in the curve case, gives a one-to-one correspondence  between irreducible flat vector bundles and stable Higgs bundles with Chern classes vanishing in $H^{2m} (X, \QQ )$ for $m\ge 1$. This result is a generalization of the theorem of Narasimhan--Seshadri in the curve case and Donaldson, Uhlenbeck--Yau in higher dimensions, saying that irreducible unitary flat vector bundles are the same as stable vector bundles with vanishing rational Chern classes. However, unlike the above result, Simpson's correspondence usually changes an algebraic structure on the underlying topological vector bundle. Therefore one can also consider analogs of Bloch's conjecture (and its weaker but more precise version in $B^*(X)$) for semistable Higgs bundles with vanishing rational Chern classes.  

In case of usual stable bundles (with the zero Higgs field) these two versions of conjectures (flat and Higgs one) are equivalent. This is a special case of the fact that Chern classes of complex polarizable variations of Hodge structure are the same as for the correspoding Higgs bundles (since the Higgs bundle is obtained as the associated graded for the Hodge filtration). In particular, the conjectures become the same, e.g., for flat bundles arising from the Gauss-Manin connection on the cohomology group of a family of smooth projective varieties.
However, it is not clear if they are equivalent in general (mainly because one does not know existence
of the Bloch--Beilinson filtration). Still one can prove equivalence of weak versions of Bloch's conjecture for flat and Higgs bundles (see Proposition \ref{equivalence}).

One can also prove that various versions of the above conjectures hold for fields of positive characteristic (see Subection \ref{positive-char}). The most useful result in this direction (from the point of view of complex varieties) is the following  theorem (see Remark \ref{Bloch-in-char-p-MIC}):

\begin{theorem}\label{main3}
	Let $X$ be a smooth projective variety of dimension $n$ defined over an algebraically closed field $k$ of positive characteristic $p$. Let $D$ be a simple normal crossing divisor on $X$ and let us assume that $(X, D)$ is liftable to $W_2(k)$. Let us also fix an ample divisor $H$ on $X$.
	Let $(V, \nabla)$ be a slope $H$-semistable bundle of rank $r\le p$ with a logarithmic connection  on $(X, D)$.
	If  $\int _X\Delta (V)\cdot H^{n-2}=0$ then for every $i\ge 1$
	we have $r^i c_i(V) = \binom{r}{i} c_1(V)^i$
	in $B^i (X)_{\QQ}$. Moreover, if $k$ is an algebraic closure of a finite field then equalities $r^i c_i(V) = \binom{r}{i} c_1(V)^i$ hold in $\CH^i (X)_{\QQ}$.
\end{theorem}

In the above theorem $\Delta (V)= 2rc_2(V)-(r-1)c_1^2(V)$ is the discriminant of $V$. The semistability condition is necessary even id $D=0$ as in positive characteristic there is essentially no bound on Chern classes of flat vector bundles (e.g., the Frobenius pullback of any coherent $\cO_X$-module acquires an integrable connection).

 The above theorem opens the road for the proof of (weak version of) Bloch's conjecture using reduction to positive characteristic. Unfortunately, this again leads to some questions that seem to be very difficult in general (especially if the answers are positive; see Subsection \ref{reductions?} for details).

\medskip

The paper is organized as follows. In Section 1 we give various preliminaries and prove some auxiliary results.
Section 2 is devoted to proof of various generalizations of Theorems \ref{main} and \ref{main2}. We also formulate various conjectures that were suggested by other conjectures of Beilinson, Bloch, Esnault and Simpson. 
In Section 3 we study Higgs version of Bloch's conjecture and ask some questions about reductions of Chow's ring and of the $B$-ring.

\subsection*{Notation}

Let $X$ be a variety over an algebraically closed field $k$. In this paper a vector bundle is a locally free coherent sheaf of $\cO_X$-modules.

If $X$ is smooth we treat Chern classes of coherent sheaves on $X$ as elements of the Chow ring $\CH^*(X)$. 
Here $\CH ^i(X)$ denotes the group of codimension $i$ algebraic cycles on $X$ modulo rational equivalence.
In this case the first Chern class $c_1: \Pic X\to \CH^1(X)$ is an isomorphism. 

If $E$ is a  rank $r$ coherent sheaf on $X$ then by  $\Delta (E)$ we denote the \emph{discriminant} of $E$.
It is defined as the class $2rc_2(E)-(r-1)c_1^2(E)$ in the Chow group $\CH^2(X)$. 

By $B^i(X)$ we denote the group of codimension $i$ algebraic cycles on $X$ modulo algebraic equivalence. 
We can also form a ring $B^* (X)=\bigoplus B^i (X)$, which comes with a canonical map from $\CH ^* (X)$.

In the following we denote by $\SL (r)$,  $\PGL(r)$ and $\GL(r)$ algebraic groups (or algebraic group schemes)
and by  $\SL (r, k)$,  $\PGL(r, k)$ and $\GL(r, k)$, respectively,  the groups of their $k$-points.

A map $f: Y\to X$ between smooth algebraic $k$-varieties is called an \emph{alteration} if it is proper and generically finite.

\section{Preliminaries}

\subsection{Bloch's conjecture}

The Bloch and Beilinson conjecture predicts that for a smooth complex projective variety $X$ there exists a functorial filtration $$0=F^{i+1}\subset F^i \subset \dots \subset F^1 \subset F^0=\CH^ i (X) _{\QQ} ,$$
which is compatible with products and correspondences (see \cite{Ja1}, \cite{Ja2} and \cite{Vo} for various candidates and properties of this filtration).  In  \cite[Introduction]{Bl2} (see also  \cite[3.1]{Es1}) 
Bloch conjectured the following:

\begin{conjecture} \emph{(Bloch's conjecture)} \label{Bloch's-conjecture}
	If $(V, \nabla)$  is flat then  $c_i (V)\otimes \QQ \in F^i \CH^i (X)_{ \QQ}$ for all $i\ge 1$.
\end{conjecture}

Let us recall that $F^1$ should be the kernel of $\CH^i (X)_{\QQ } \to H^{2i} (X, \CC)$ and
$F^2$ should be the kernel of $\CH^i (X)_{ \QQ }\to H^{2i} _{\cD} (X, \QQ  (i))$. If this  holds then
Bloch's conjecture holds  for $i=1$ since de Rham's Chern classes of a flat vector bundle vanish  and for
$i=2$ by  Reznikov's theorem \cite[Theorem 1.1 and 2.7]{Re}.

\medskip

Let us recall that $B^i(X)$ is the group of codimension $i$ algebraic cycles on $X$ modulo an algebraic equivalence. We have canonical maps $\CH^i (X)\to B^i (X)$, so for every coherent sheaf $E$ on $X$ we can talk about Chern classes $c_i (E) $ in $ B^i (X)$ defined as images of $c_i (E)\in \CH^i (X)$ (by abuse of notation we use $c_i$
in both cases).
Let us recall that existence of the Bloch--Beilinson filtration and the standard conjecture B imply that cycles in the lowest stratum of the Bloch--Beilinson filtration are algebraically equivalent to zero (see \cite[p.~226--227]{Ja2}). In particular, we can consider the following weak form of Bloch's conjecture:

\begin{conjecture} \label{weak-Bloch}
	If $(V, \nabla)$  is flat then  $c_i (V)$ are torsion in $B^i (X)$ for all $i\ge 1$.
\end{conjecture}

Even this weak form of Bloch's conjecture is still unknown. 

\subsection{The Riemann--Hilbert correspondence}

Let $X$ be a smooth complex quasi-projective variety with a fixed basepoint $x\in X(\CC)$. 
Let $\rho: \pi_1 (X^{\an},x)\to \GL(r, \CC)$ be a representation of the topological fundamental group of $X^{\an}$
and let $\VV_{\rho}$ denote the corresponding local system on $X^{\an}$.
We have an associated flat analytic vector bundle $(V^{\an}_{\rho}, \nabla^{\an} _{\rho})$ defined by $V^{\an}_{\rho}=\cO _{X^{\an}}\otimes _{\CC} \VV _{\rho}$ and  $ \nabla^{\an}_{\rho}=d\otimes \id _{\VV _{\rho}}$.
On $X$ there exists a unique (up to an isomorphism) algebraic vector bundle $V_{\rho}$ 
with an algebraic integrable connection $\nabla_{\rho}$ such that $(V_{\rho},\nabla_{\rho})$ has regular singularities at infinity and its analytification $(V_{\rho},\nabla_{\rho})^{\an}$ is isomorphic to $(V^{\an}_{\rho},\nabla ^{\an}_{\rho})$.

Moreover, if $(V,\nabla)$ is a vector bundle on $X$  with an integrable connection and  regular singularities at infinity then it is isomorphic to  $(V_{\rho},\nabla_{\rho})$, where $\rho$ is the monodromy representation of 
$(V^{\an},\nabla ^{\an})$. In fact, we have an equivalence of categories between the category of vector bundles with an integrable connection and regular singularities on $X$ and local systems on $X^{\an}$ (see, e.g., \cite[Corollary 5.3.10]{HTT}).
This is called (Deligne's) Riemann--Hilbert correspondence. We will need the following functoriality of this correspondence:

\begin{lemma}\label{functorial-RH}
	Let $f: (X, x)\to (Y,y)$ be a morphism of smooth complex quasi-projective varieties and let $f_*: \pi_1 (X^{\an},x)\to
	\pi_1 (Y^{\an},y)$ be the induced map. If $\rho: \pi_1 (Y^{\an},y)\to \GL(r, \CC)$ is a representation then
	$$ f^*(V_{\rho},\nabla_{\rho})\simeq (V_{\rho\circ f_* },\nabla_{\rho\circ f_*}).$$
\end{lemma}

\begin{proof}
This can be seen as a special case of functoriality for $D$-modules.
More precisely, Deligne's Riemann--Hilbert correspondence can be extended to an equivalence
of the categories of regular holonomic $D_X$-modules and perverse $\CC_{X^{\an}}$-modules with algebraically constructible cohomology (see \cite[Theorem 7.2.5]{HTT}). Commutativity of the  derived pull-back functor on $D$-modules with the Riemann--Hilbert correspondence follows from \cite[Theorem 7.1.1]{HTT}. So we need only to check  that the derived pull-back of $D$-modules agrees with the usual pull-back of flat vector bundles.
But the characteristic variety of a flat vector bundle is the zero section, so $f$ is non-characteristic with respect
to any flat vector bundle (see \cite[Definition 2.4.2]{HTT}). So the required assertion follows from  \cite[Theorem  2.4.6]{HTT}.   
\end{proof}

\subsection{Geometric monodromy groups}

Let $\Gamma$ be a group and $V$ a finite dimensional complex vector space. We say that a representation $\rho: \Gamma \to \GL (V)$ is \emph{Lie irreducible} if it is irreducible on all finite index subgroups of $\Gamma$. 
The \emph{geometric monodromy group $M_{\rm geom}(\rho)$} is the Zariski closure of $\rho (\Gamma)$ in $\GL (V)$, i.e., the smallest Zariski closed subgroup of $\GL (V)$, whose $\CC$-points contain the image of $\rho$.
A representation $\rho$ is Lie irreducible if and only if $V$ is simple as a module over the connected component $M_{\rm geom}(\rho)^{0}$ of $M_{\rm geom}(\rho)$.   

\begin{lemma}\label{Lie-irred-lemma}
	Let $X$ be a smooth complex quasi-projective variety and let $\rho: \pi_1 (X^{\an},x)\to \GL(r,\CC)$ be a representation. Then the following conditions are equivalent:
	\begin{enumerate}
		\item $\rho$ is not Lie irreducible.
		\item There exists a finite \'etale covering $f:(Y,y)\to (X,x)$ such that $f^*\rho: \pi_1 (Y^{\an},y)\to \GL(r,\CC)$ is not irreducible.
		\item There exists an alteration $f:(Y,y)\to (X,x)$ such that $f^*\rho: \pi_1 (Y^{\an},y)\to \GL(r,\CC)$ is not irreducible.
	\end{enumerate}
\end{lemma}

\begin{proof}
	Assume 1 and let $H \subset  G=\pi_1 (X^{\an},x)$ be a finite index subgroup such that $\rho|_{H}$ is not irreducible. 
	Standard arguments show that  exists a finite index normal subgroup $H'$ in $G$, which is contained in $H$. Namely, we can set $H'=\ker \varphi$, where $\varphi: G\to \Sigma_{G/H}$ comes from the left $G$-action on the set of left cosets of $H$ in $G$.  
	Clearly, the restriction  $\rho|_{H'}$ is not irreducible. Then as  $f$ in 2 we can take the \'etale covering defined by 
	the quotient $\pi_1 (X^{\an},x) \to \pi_1 (X^{\an},x)/H'$.
	
	Clearly 2 implies 3. To see that 3 implies 1 it is sufficient to note that if $f$ is an alteration then the image of 
	$\pi_1 (Y^{\an},y)\to \pi_1 (X^{\an},x)$ has finite index.
\end{proof}

\medskip

If $(V, \nabla)$ is a flat vector bundle on $X$ then the \emph{geometric monodromy group} $M_{\rm geom}(V, \nabla , x)$ of $(V, \nabla)$ at $x\in X(\CC)$ is the geometric monodromy group of the associated monodromy representation 
$\rho: \pi_1(X,x)\to \GL (r, \CC)$.

\begin{lemma} \label{connected-components-monodromy}
	Let $f: (Y,y)\to (X,x)$ be an alteration of pointed smooth quasi-projective complex varieties.
	Then for any flat vector bundle $(V, \nabla)$ on $X$ with regular singularities
	the induced map of connected components $$M_{\rm geom}^0(f^*(V, \nabla), y)\to M_{\rm geom}^{0}(V, \nabla , x)$$
	is an isomorphism. 
\end{lemma}

\begin{proof}
	The image of  $f_*:\pi_1 (Y^{\an},y)\to \pi_1 (X^{\an},x)$ has finite index.
	So by Lemma \ref{functorial-RH} it is sufficient to prove that if $G \subset  \pi_1 (X^{\an},x)$ is a subgroup of finite index and $\rho :  \pi_1 (X^{\an},x) \to \GL(r)$ is a representation then 
	$M_{\rm geom}(\rho |_G)^{0}= M_{\rm geom}(\rho )^{0}$. But $G$ contains a subgroup $H$ which is normal and of finite index in $\pi_1 (X^{\an},x)$. This subgroup corresponds to a finite \'etale morphism $g: Y'\to X$, which is 
	Galois with group $\pi_1 (X^{\an},x)/H$. Then equality $M_{\rm geom}(\rho |_{H})^{0}= M_{\rm geom}(\rho)^{0}$
	follows from  Lemma \ref{functorial-RH} and \cite[Proposition 1.4.5]{Ka1} (one needs also to use \cite[Proposition 5.2]{Ka0}). This implies the required equality $M_{\rm geom}(\rho |_G)^{0}= M_{\rm geom}(\rho )^{0}$.
\end{proof}

\subsection{Character varieties} \label{character}

Let $\Gamma$ be a finitely generated group and let $G$ be an affine group scheme of finite type over an algebraically closed field $k$. 
There exists  an affine $k$-scheme  $R(\Gamma, G)$ representing the functor, which to a $k$-scheme $S$ associates the set 
$\Hom (\Gamma, G(H^0(S, \cO_S )))$. It comes with a natural $G$-action given by conjugation. If $G$ is reductive we can form the geometric invariant theory quotient $ R(\Gamma, G)/\!\!/ G:= \Spec k[R(\Gamma, G)]^G $, which is called the \emph{character scheme}. This scheme is often called a ``character variety'' but we refrain from this terminology as it is a scheme, which is rarely a variety. Since  $R(\Gamma, G)$ is affine, $ R(\Gamma, G)\to R(\Gamma, G)/\!\!/ G $ is a uniform categorical quotient (a universal categorical quotient if $\chr k=0$).

We say that a representation $\rho : \Gamma \to G(k)$ is \emph{Zariski dense} if the image of $\rho $ is Zariski dense
in $G$ (since $k$ is algebraically closed this is equivalent to $\rho (\Gamma)$ being dense in $G(k)$ with the topology induced by the Zariski topology on $G$).

\begin{proposition}\emph{(see \cite[Proposition 8.2]{AB})} \label{density}
	Assume that $G$ is a connected, semisimple group over $k=\CC$. Then there exists a Zariski open subset 
	$R_{\rm Zd} (\Gamma, G)\subset R (\Gamma, G)$ such that the representation $\rho: \Gamma \to G(\CC)$ 
is Zariski dense if and only if the $\CC$-point corresponding to $\rho$ lies in $R_{\rm Zd} (\Gamma, G)$.
\end{proposition}

By abuse of notation  we do not distiguish between an algebraic group and the functor of points of the associated group scheme.

\subsection{Betti moduli spaces}

Let $X$ be a smooth complex quasi-projective variety with a fixed closed point $x\in X(\CC )$.
Let $G$ be a reductive group over $\CC$.  We set $R(X, x, G):=R(\pi_1 (X^{\an},x), G)$. The corresponding
character variety $M_B(X, G):=R(X, x, G)/\!\!/ G$ is called the \emph{Betti moduli space} of $X$ (with values in $G$). It corepresets the functor of isomorphism classes of representations of $\pi_1 (X^{\an},x)$ in $G$ (see Subsection \ref{character}; it also corepresents the functor of $G$-local systems as in \cite{Si}).
In case $G=\GL(r)$ we write  $R(X, x, r)$ and $M_B (X, r)$ instead of $R(X, x, G)$ and $M_B(X, G)$, respectively.

Let us fix a compactification $\bar X$ of $X$, which is smooth projective and such that $D=\bar X-X$ is a simple normal crossing divisor. For each irreducible component $D_i$ of $D$ let us choose $\gamma _i\in \pi_1 (X^{\an},x)$ that is represented by a simple loop around that component.

For each $D_i$ let us choose a closed (reduced) ${\rm Ad} (G)$-invariant subscheme $C_i\subset G$.
There exists an affine $\CC$-scheme  $R(X, x,G, \{ C_i \})$ representing the functor, which to a $k$-scheme $S$ associates the subset of $\Hom (\Gamma, G(H^0(S, \cO_S )))$ of representations $\rho: \Gamma \to G(H^0(S, \cO_S ) )$ such that
$\rho (\gamma _i) \in C_i (H^0(S, \cO_S ))$. It is a closed $G$-invariant subscheme of  $R(X, x, G)$, so we can form 
the geometric invariant theory quotient $M_B (X, G,  \{ C_i \}) =R(X, x,G, \{ C_i \}) //G$, which is a closed subscheme
of  $M_B (X, G)$.

Let us recall that an element $g\in G(\CC)$ is called \emph{quasi-unipotent} if there exists a positive integer $n$ such that $g^n$ is unipotent (this makes sense in an arbitrary linear algebraic group due to existence of multiplicative Jordan decomposition in such groups). 

A representation $\rho: \pi_1 (X^{\an},x) \to G (\CC)$ has \emph{quasi-unipotent monodromy at infinity} if each $\rho (\gamma _i)$ is quasi-unipotent. For such a representation let us define $C_i\subset G$ as the closure of a conjugacy class of $\rho (\gamma _i)$. Then we say that $\rho$ is \emph{rigid} if it is an isolated point of the moduli space $M_B (X, G,  \{ C_i \})$. Note that rigidity of $\rho$ depends on the choice of the compactification $\bar X$ but once we fix it, the classes $\{C_i\} $ are uniquely  determined by $\rho$.

Let us recall that by the monodromy theorem, the monodromy representation of a family of smooth projective varieties has quasi-unipotent monodromy at infinity. 

\medskip

\medskip

If   $\rho: \pi_1(X,x)\to \GL (r, \CC)$  has {quasi-unipotent monodromy at infinity} then  we say that it is \emph{properly rigid}, if $M_{\rm geom}(\rho)$ is reductive and the induced homomorphism $\tilde \rho: \pi_1(X,x)\to  M_{\rm geom}(\rho)$ is rigid
(note that Jordan decomposition is preserved under homomorphisms of algebraic groups, so $\tilde \rho$
has also quasi-unipotent monodromy at infinity). Again, proper rigidity depends not only on $\rho$ but also on the choice of $\bar X$. We will ususally ignore this dependence.

\begin{theorem}\label{rigid-VHS}
Let $\rho: \pi_1(X,x)\to \GL (r, \CC)$ be a properly rigid representation with quasi-unipotent monodromy at infinity.
Then $\rho$ underlies a complex variation of Hodge structures.
\end{theorem}

\begin{proof}
In case $X$ is projective this is \cite[Lemma 4.5]{Si0}. The same arguments together with
\cite[Theorem 8.1]{CS} give the general assertion. Note that although \cite[Theorem 8.1]{CS}
is formulated only in the rank $2$ case, this assumption is used only in showing that the corresponding 
$\CC$-VHS has weight $1$.
\end{proof}

\subsection{De Rham moduli spaces on projective varieties}

Let $X$ be a smooth complex projective variety with a fixed closed point $x\in X(\CC )$.
Let  $R_{DR} (X,x, r )$ be the $\CC$-scheme, representing the functor which to a $\CC$-scheme $S$ associates the set of isomorphism classes of $(V, \nabla, \alpha)$, where $(V, \nabla)$ is a rank $r$ vector bundle with a relative integrable connection on $X\times S/S$ and  $\alpha: V|_{\{x\} \times S }\mathop{\to}^{\simeq} \cO_S^{r}$ is a frame along the section.
Let us recall that $M_{DR}(X, r)$ is constructed as the geometric invariant theory quotient of  $R_{DR} (X,x, r )$ by $\GL (r, \CC)$ and it is also its universal categorical quotient
(see \cite[Theorem 6.13]{Si}). If two $\CC$-points of $R_{DR} (X,x, r)$ corresponding to triples
$(V, \nabla, \alpha)$ and  $(V', \nabla ', \alpha ')$ are mapped to the same point of $M_{DR}(X, r)$
then there exist filtrations of $V$ and $V'$ such that the associated graded vector bundles are isomorphic.
Therefore the maps
$c_i: (R_{DR} (X,x,  r)(\CC ) \to \CH^i (X),$
induced by Chern classes, descend to  
$c_i: M_{DR} (X, r)(\CC ) \to \CH^i (X).$

Now let $M_{DR}(X, G)$ be the coarse moduli space of principal (right) $G$-bundles on $X$ together with an integrable connection
(see \cite[Theorem 9.10 and an analogue of Proposition 9.7 in the de Rham case]{Si}).

The Riemann--Hilbert correspondence shows that the analytification of  $M_B(X, G)$ is (analytically) isomorphic to the analytification of the de Rham moduli space $M_{DR}(X, G)$ (see \cite[Theorem 9.11]{Si}). So Theorem \ref{main2}
(in case of rank $2$ for simplicity) can be somewhat imprecisely reformulated as saying that if $T$ is an irreducible component of $M_{DR} (X, \PGL(2))$ of dimension $\ge 1$ containing a flat vector bundle with Zariski dense monodromy representation then the image of
$4c_2-c_1^2: M_{DR} (X, \PGL(2)) (\CC)\to  \CH^2 (X)_{}$ lies in $N$-torsion classes for some $N$.

\subsection{De Rham side for quasi-projective varieties}

\begin{proposition}\label{map-for-de_Rham-families}
	Let $X$  be a smooth complex analytic space and let $S$ be an analytic space. Let $(V, \nabla)$ be a  rank $r$ (holomorphic) vector bundle with a relative holomorphic connection on $X\times S /S$. Let $x$ be a point of $X$ and assume that 
	there exists an isomorphism $\alpha : i_x^{*}V\simeq \cO_{S}^r$, where $i_x: S\to X\times S$ denotes the canonical embedding of $S$ onto ${\{x\} \times S}$. Then there exists an analytic morphism 
	$\varphi _{\alpha} : S\to R(X,x, r)^{\an}$ such that for every $s\in S$ the point $\varphi _{\alpha} (s)$ corresponds to the monodromy representation of $(V_s, \nabla _s)$.
\end{proposition}

\begin{proof}
	Let $p_S: X\times S\to S$ denote the canonical projection.
	By \cite[Lemma 7.6]{Si} the sheaf $V^{\nabla}$ of sections $v$ of $V$ such that $\nabla (v)=0$ is a locally free sheaf of 
	$p_S^{-1}(\cO_S)$-modules. As in the proof of \cite[Lemma 7.7]{Si} the composition 
	$$\beta: i_x^{-1} (V^{\nabla}) \to i_x^{-1} (V)\to i_x^*(V)\mathop{\to}^{\alpha} \cO_S^r$$
	gives a frame of $V^{\nabla}$ over $i_x$. Now \cite[Lemma 7.3]{Si} works also for any smooth analytic space with a fixed point $(X,x)$, so the pair $(V^{\nabla}, \beta)$ gives rise to the required morphism. 
\end{proof}

\begin{corollary} \label{moduli-map-for-de_Rham-families}
		Let $X$  be a smooth complex analytic space and let $S$ be an analytic space. Let $(V, \nabla)$ be a  rank $r$ (holomorphic) vector bundle with a relative holomorphic connection on $X\times S /S$. Then there exists an analytic morphism $\varphi  : S\to M_B(X, r)^{\an}$ such that for every $s\in S$ the point $\varphi  (s)$ corresponds to the class of the monodromy representation of $(V_s, \nabla _s)$.
\end{corollary}

\begin{proof}
Let us fix a point $x\in X$. Then $i_x^*V$ is a vector bundle on $S$ so for every point $s\in S$ we can find an open neighbourhood  $s\in U_s\subset S$ and an isomorphism $\alpha_s:  i_x^{*}V\simeq \cO_{U_s}^r$. Then the maps
 	$ U_s\to M_B(X,r)^{\an}$ obtained by composing $\varphi _{\alpha _s}$ with the quotient map  $R(X,x, r)^{\an}\to M_B (X, r)^{\an}$ glue together to the required map $\varphi$. 
\end{proof}

\subsection{Algebraic fundamental groups and monodromy groups} 

Let $X$ be a smooth geometrically connected quasi-projective variety defined over a field $k$ of characteristic zero.
Let $\MIC  (X/k)$ be the category of locally free $\cO_X$-modules of finite rank with an integrable connection (local freeness follows from coherence of an $\cO_X$-module). It is an abelian, $k$-linear, rigid tensor category.  A rational point $x\in X(k)$ defines a fiber functor $\omega_x: \MIC   (X/k)\to \Vect _k$, $\omega _x (V, \nabla) =V\otimes k(x)$ to the category of finite dimensional $k$-vector spaces. This gives rise to a $k$-group scheme $\pi _1 ^{\alg}(X, x) ={\mathop{\rm Aut} } ^{\otimes} (\omega_x)$. Tannaka duality says that $\MIC  (X/k)$ is equivalent via $\omega_x$ to the representation category 
of   $\pi _1^{\alg} (X, x) $. For any object $(V, \nabla) \in \MIC (X/k) $ we define its \emph{algebraic monodromy group} $M_{\alg} (V, \nabla, x)$ to be the $k$-affine group scheme Tannaka dual to the full Tannakian subcategory $\langle 
(V, \nabla)\rangle$ of 
$\MIC   (X/k)$ spanned by $(V, \nabla)$. By \cite[II, Proposition 2.21 a)]{DMOS} this group scheme is the image of 
$\pi _1 ^{\alg}(X, x)$ in $\GL(V\otimes k(x))$.

Similarly, we can define the category  $\MIC ^{\rs} (X/k)$ as the full subcategory of $\MIC (X/k)$ whose objects are regular singular at infinity. Again this is an abelian, $k$-linear, rigid tensor category and any rational point  $x\in X(k)$
defines a fibre functor, which leads to  $\pi _1^{\alg , \rs} (X, x)$. Note that if $(V, \nabla) \in \MIC ^{\rs}  (X/k) $
then $M_{\alg} (V, \nabla, x)$ is Tannaka dual to the full Tannakian subcategory of 
$\MIC ^{ \rs}  (X/k)$ spanned by $(V, \nabla)$. This follows from the fact that a subobject and quotient of an object of   $\MIC ^{ \rs} (X/k)$ in  $\MIC (X/k)$ still lies in  $\MIC ^{\rs} (X/k)$.

The group $M_{\alg} (V, \nabla, x)$ is sometimes called the differential Galois group of $(V, \nabla)$ (see, e.g., \cite[IV]{Ka0}).  In case $k=\CC$ and $(V, \nabla)$ is a vector bundle with an integrable connection and regular singularities with monodromy representation  $\rho$ then $M_{\alg} (V, \nabla, x)$ coincides with
$M_{\rm geom} (V, \nabla, x)$  (see \cite[Proposition 5.1 and Proposition 5.2]{Ka0}).

Let us note the following folklore fact:

\begin{lemma}
Let $K\supset k$ be a field extension. Then the natural base change morphisms 
$$\pi _1 ^{\alg}(X\otimes _kK, x\otimes _kK) \to \pi _1 ^{\alg}(X, x)\otimes _kK$$
and 
$$\pi _1 ^{\alg , \rs}(X\otimes _kK, x\otimes _kK) \to \pi _1 ^{\alg , \rs}(X, x)\otimes _kK$$
are closed immersions.
\end{lemma} 

\begin{proof}
	By \cite[II, Proposition 2.21 (b)]{DMOS} we need to show that every object $M_K$ of $\MIC  (X\otimes _k K/K)$ is isomorphic to a subquotient of an object coming from the representation of $\pi _1 ^{\alg}(X, x)\otimes _kK$. But  $M_K$ is defined over a finite extension $k\subset L$ for some $L\subset K$. Let $M_L$ denote the corresponding integrable connection with $M_K\simeq M_L\otimes _LK$ and let $\varphi : X_L\to X$ denote the base change map.	Then $M_L$ is a subquotient of $\varphi ^* \varphi_*M_L$, which comes from $X$. So $M$ is also isomorphic to a subquotient	of an integrable connection coming from $X$. The proof in the second case is analogous.
\end{proof}

In general, it is not true that the above base change morphisms are isomorphisms.
However, we have the following lemma:

\begin{lemma}\label{surjectivity-after-extension}
	Let $K$ be an algebraic closure of a field $k$. Let $x\in X(k)$ be a $k$-rational point and let $M=(V, \nabla) \in
	\MIC  (X/k)$. Then there exists a finite field extension $k\subset l$ for some $l\subset K$ such that
	the natural base change morphism 
		$$M_{\alg}(M\otimes _k K, x\otimes _k K) \to M _{\alg}(M\otimes _k l,  x\otimes _kl )\otimes _lK$$
is an isomorphism.
\end{lemma}

\begin{proof}
	By the previous lemma for any $k\subset l\subset K$ 
	$$M_{\alg}(M\otimes _k K, x\otimes _k K) \to M _{\alg}(M\otimes _k l,  x\otimes _kl )\otimes _lK$$
	is a closed immersion, so we need only to find $l$ for which this morphism is faithfully flat.
	
Since $M_{\alg}(M\otimes _k K, x\otimes _k K)$ is algebraic it has a finite dimensional faithful representation (see \cite[II, Corollary 2.5]{DMOS}), which corresponds to 
some element $M'\in \MIC  (X\otimes _k K /K)$. By definition $M'$ is an element of the full Tannakian subcategory of 
$\MIC   (X\otimes _k K/K)$ spanned by $M\otimes _k K$, so it is a subquotient of $P (M\otimes _k K, M^{\vee}\otimes _kK)$ for some $P\in \NN [t,s]$. But such $M'$ is defined over a finite extension $k\subset l$, i.e., there exists  
$M''\in \MIC  (X\otimes _k l/l)$ such that $M'\simeq M''\otimes _lK$. 
The category of rational finite dimensional representations of $M _{\alg}(M\otimes _k l,  x\otimes _kl )\otimes _lK$ is equivalent to the full subcategory $\cC$ of $\MIC  (X\otimes _k K /K)$, such that $N$ is an object of $\cC$ if and only if there exists $N'\in \langle M\otimes _k l \rangle$ such that $N$ is a subobject of $N'\otimes _lK$. 
Since $M'$ is an object of $\cC$, \cite[II, Proposition 2.21 (a)]{DMOS} implies that
$$M_{\alg}(M\otimes _k K, x\otimes _k K) \to M _{\alg}(M\otimes _k l,  x\otimes _kl )\otimes _lK$$
is faithfully flat.
\end{proof}

\begin{remark}
In fact, by \cite[Proposition 1.3.2]{Ka1} the above lemma holds for $l=k$. However, Lemma \ref{surjectivity-after-extension} is sufficient for our applications.
\end{remark}

Let us also recall the following Gabber's theorem (see \cite[Specialization Theorem 2.4.1]{Ka}):

\begin{theorem} \label{Gabber}
Let $k$ be a field of characteristic zero and let $R=k \llbracket t \rrbracket$, $S= \Spec R$.
Let $X/S$ be a smooth separated $R$-scheme of finite type  with geometrically connected fibers and let $x\in X(S)$.
Let  $(V, \nabla)$ be a vector bundle on $X$ with a relative (over $S$) integrable connection $\nabla$. Let $\eta $ be the generic point of $S$ and $s$ the special point of $S$. Let $G/ S$ be the closed $S$-flat subgroup scheme of $\GL (x^*V)$ obtained as the schematic closure of $M_{\alg} (V _{\eta}, \nabla _{\eta}, x (\eta) )$. Then 
$$M_{\alg} (V _{s}, \nabla _{s}, x (s) ) \subset G_s .$$
\end{theorem}

\section{Flat bundles on complex quasi-projective varieties}

Let $X$ be a smooth complex quasi-projective variety with a fixed basepoint $x\in X(\CC)$ and 
let $\rho: \pi_1 (X^{\an},x)\to \GL(r, \CC)$ be a representation. Then we can define Chern classes of the representation $\rho$  by the formula $c_i (\rho):=c_i (V_{\rho})\in \CH^i (X)$.

Let $(V, \nabla)$ be a vector bundle with an integrable connection on $X$.
Let us recall that the Chern class  $c_i(V)$ is mapped under the cycle class map  $\CH ^i (X ) \to H^{2i}_{\cD} (X, \ZZ (i))$ to the Chern class $c_i ^{D}(V)$ of $V$ in the Deligne-Beilinson cohomology.
Cheeger and Simons defined also characteristic classes 
$$\hat c _i (V, \nabla)\in H^{2i-1} (X, \CC/\ZZ (i)).$$
These classes are mapped to $c_i ^{D}(V)$ under the natural homomorphism
$$H^{2i-1} (X, \CC/\ZZ (i))\to H^{2i}_{\cD} (X, \ZZ (i))$$
(see \cite[Theorem 1]{DHZ}). Let us also recall that one can define characteristic classes $c _i ^{\an} (V, \nabla)\in H^{2i-1} (X, \CC/\ZZ (i))$ by universality (see \cite[3.4]{Es1} for precise references). These classes are very similar to $\hat c _i (V, \nabla)$ but it seems to be still unknown if they coincide in general (equality of these classes is equivalent to \cite[Conjecture 4]{DHZ}). Note that unlike the Chern classes in the Chow group, the classes
$\hat c _i (V, \nabla)$, $c _i ^{\an} (V, \nabla)$ and $c_i ^{D}(V)$ do not depend on the algebraic structure of  $(V, \nabla)$ but only on its analytic structure $(V^{\an }, \nabla ^{\an })$.

\medskip

\subsection{Chern classes of rigid representations}

\begin{proposition} \label{rigid}
	Let  $\rho: \pi_1 (X^{\an},x)\to \GL(r, \CC)$ be a properly rigid representation with quasi-unipotent monodromy at infinity. Then $\hat c_i(V_{\rho},\nabla_{\rho})$ and $c _i ^{\an} (V_{\rho},\nabla_{\rho})$ vanish
	in $H^{2i-1} (X^{\an}, \CC /\QQ)$ for all $i\ge 1$. 
\end{proposition}

\begin{proof}
	Since $\rho$ is rigid, we can find an algebraic number field $K$, an embedding $\eta : K\to \CC$ 
	and a representation  $\rho _K: \pi_1 (X^{\an},x)\to \GL(r, K)$ such that $\rho _K\otimes _{K, \eta } \CC= \rho$.
	The idea behind this fact is that since $ \pi_1 (X^{\an},x)$ is finitely generated, we can find a subfield $L$ of $\CC$
	that is finitely generated over $\QQ$ such that $\rho$ is defined over $L$. Considering various embeddings of $L$ into $\CC$ 
	we get a family of representations containing $\rho$. Now rigidity of $\rho$ implies that for every $\gamma \in \pi_1 (X^{\an},x) $ traces of images of $\gamma$ take only finitely many values. This easily shows that $\rho (\gamma)$ is algebraic, so we can find the required field $K$.
	
	By Theorem \ref{rigid-VHS} and \cite[Lemma 6.6]{CS}	for any $\sigma : K\to \CC$ the induced local system $V_{\sigma \circ \rho_K }$ underlies a polarized complex variation of Hodge structure. Therefore the required assertions follow from \cite[Theorem A.1]{CE} and \cite[Theorem A.4]{CE}.
\end{proof}

Note that the above proposition implies that any  rigid representation of rank $1$  with quasi-unipotent monodromy at infinity is of finite order (cf. \cite[Lemma 7.5]{LS}).

The above result motivates the following conjecture:

\begin{conjecture}\label{rigid-conjecture}
	If $\rho: \pi_1 (X^{\an},x)\to \GL(r, \CC)$ is a properly rigid representation with quasi-unipotent monodromy at infinity then the associated algebraic flat vector bundle $(V_{\rho}, \nabla _{\rho} )$ has torsion $c_i (V_{\rho})$ in $\CH ^i (X)$ for all $i\ge 1$.
\end{conjecture}

Taking a finite covering we can reduce 
the study of  representations of $\pi_1 (X^{\an},x)$ with quasi-unipotent monodromy to the unipotent case
(however, by passing to a finite covering the representation induced from properly rigid need not be properly rigid).
In this case  there is a stronger version of the above conjecture. Namely, 
let us assume that $\rho$  has unipotent monodromy at infinity and
fix a projective completion $X\subset \bar X$ such that $D=\bar X-X$ is a simple normal crossing divisor on $\bar X$. 
Then there exists a canonical Deligne's extension of  $(V_{\rho}, \nabla _{\rho} )$  to a vector bundle $\bar V _{\rho}$ with a logarithmic 
connection $\bar \nabla  _{\rho}: \bar V _{\rho}\to \bar V _{\rho}\otimes \Omega_{\bar X} (\log D)$ such that the residues of $\bar \nabla _{\rho}$ along the irreducible components of $D$ are nilpotent. 
Then  the eigenvalues of the residues of $\bar \nabla _{\rho}$ are equal to $0$ and
a standard computation of Chern classes in the de Rham cohomology of $\bar X$ shows that $c_i^{DR} (\bar V _{\rho})=0$ for $i\ge 1$.

Summing up, one can expect that if $\rho: \pi_1 (X^{\an},x)\to \GL(r, \CC)$ is a properly rigid representation with unipotent monodromy 
at infinity then the Chern classes  $c_i (\bar V _{\rho})$ are also torsion in $\CH ^i (\bar X)$.

\subsection{Chern classes of general rank $2$ representations}

Conjecture \ref{rigid-conjecture} would follow from Simpson's conjecture saying that (properly) rigid local systems are of geometric origin and the following strenghtening of Esnault's conjecture on Chern classes of Gauss--Manin connections:
 
\begin{conjecture}
	If a local system $\VV$ is of geometric origin, then the associated algebraic flat vector bundle has torsion Chern classes in the Chow ring.
\end{conjecture}

Note that if $\VV$ is of geometric origin then  $\det \VV$ is a direct summand of the tensor product $\VV\otimes ...\otimes \VV$
of $\rk \VV$ copies of $\VV$. So by \cite[Proposition 2.10]{LS}  $\det \VV$ is of geometric origin. Since it has rank $1$ 
by [ibid.] it is also of finite order and hence the first Chern class of the algebraic flat vector bundle associated 
to $\VV$ is torsion. So in the above conjecture it is sufficient to consider only higher Chern classes.

\medskip

The following theorem is a rank $2$ analogue of Reznikov's theorem on Chern classes of flat bundles on smooth projective varieties (see \cite[Theorem 1.1  and 2.7]{Re}). Note that it is somewhat ``perpendicular'' to the results of Iyer and Simpson 
\cite[Proposition 1.2]{IS} that use parabolic bundles to deal with representations that are of finite order at infinity.

\begin{theorem}\label{Reznikov-analogue-rk2}
	Let  $\rho: \pi_1 (X^{\an},x)\to \GL(2, \CC)$ be a representation with quasi-unipotent monodromy at infinity
	and  let $(V , \nabla )$ be the associated  algebraic flat vector bundle. Then  $\hat c _i (V, \nabla)$ and $c_{i}^{\an}(V , \nabla )$ are torsion in  $H ^{2i-1}(X^{\an}, \CC/\ZZ (i))$ for all $i\ge 2$.
	In particular, $c_2 ^{D} (V)$ is torsion in $H^{4}_{\cD} (X, \ZZ (2))$.
\end{theorem}

\begin{proof}
	Let us consider the short exact sequence of groups
	$$1\to \mu _2\to \CC^*\to \CC^*\to 1,$$
	where the map $\CC^*\to \CC^*$ is given by $z\to z^2$.
	\cite[Lemma 2.6]{LS} shows that there exists a finite surjective morphism $\tilde X\to X$ of algebraic varieties 
	and a rank $1$ representation $\tau : \pi_1 (\tilde X, \tilde x)\to \CC^*$ such that $f^*(\det \rho )=\tau ^2$.
	Then the image of  $\tilde \rho:= f^*\rho\otimes \tau ^{-1}: \pi_1 (\tilde X, \tilde x) \to \GL(2,\CC)$ is contained in $ \SL (2, \CC)$.	Since the maps $H ^{2i-1}(X ^{\an}, \CC/\QQ (i))\to H ^{2i-1}(\tilde X^{\an}, \CC/\QQ (i))$ are injective and $\tilde \rho$ has quasi-unipotent monodromy at infinity, it is sufficient to prove the required assertions 
	assuming that $\rho$ is a  representation into $\SL(2, \CC)$.

	If the representation  $\rho: \pi_1 (X^{\an},x)\to \SL (2, \CC)$ is not irreducible then $(V, \nabla)$ has a filtration
	with quotients $(V_1, \nabla _1)$ and  $(V_2, \nabla _2)$ being  rank $1$ flat vector bundles.
	Clearly, we need to prove the required vanishing of characteristic classes only for $i=2$. Then 
	$$\hat c_{2} (V , \nabla )= \hat c_{1}(V_1 , \nabla _1) \cdot \hat c_{1} (V _2, \nabla _2)=0$$
	in $H ^{3}(X^{\an}, \CC/\QQ (2))$ and similarly for $c_2^{\an}$ (cf. proof of \cite[Theorem A.1 and Theorem A.4]{CE}).
	
	Therefore we can assume that $\rho$ is irreducible, i.e., the image of $\rho$ is not contained in any proper 
	parabolic subgroup of $\SL (2, \CC)$. We can also assume that it remains irreducible on all finite index subgroups of $\pi_1 (X^{\an},x)$ (otherwise, we can find a finite \'etale covering $f: \tilde X\to X$ on which $f^*\rho$
	is not irreducible, which as before implies the required vanishing of characteristic classes).
	
	This shows that we can assume that the image of $\rho$ is Zariski dense in $\SL (2, \CC)$ (cf. \cite[Proposition 2.7]{LS}). 
	If $\rho $ is rigid then the assertions follow from Proposition \ref{rigid}. Otherwise, by  \cite[Theorem 1]{CS} the representation $\rho$ projectively factors through an orbicurve, i.e., there exists an orbicurve $C$,
	a fibration $f:X\rightarrow C$, and a commutative diagram
	$$
	\begin{array}{ccc}
		\pi _1(X^{\an},x) & \xrightarrow{\ ~ \rho \ } & \SL (2,\CC ) \\
		\downarrow & & \downarrow \\
		\pi _1^{\orb}(C^{\an},f(x)) & \xrightarrow{\ ~ ~ \ \ } & \PGL (2,\CC ).
	\end{array}
	$$
	Composing $ \pi _1^{\orb}(C^{\an},f(x))\to \PGL (2,\CC )$ with the adjoint representation $\PGL (2,\CC ) \to \SL (3,\CC)$
	we get a representation  $\tau : \pi _1^{\orb}(C^{\an},f(x))\to \SL (2,\CC )$ such that $f^*\VV_{\tau}\simeq \Sym ^2 \VV _{\rho}$. Therefore by functoriality of the Riemann--Hilbert correspondence (Lemma \ref{functorial-RH}) we have
	$$4\hat c_2(\VV_{\rho})=\hat c_2 (\Sym^2 (\VV_{\rho} ))= f^*\hat c_2 (\VV_{\tau})= 0$$ in $H ^{3}(X^{\an}, \CC/\QQ (2))$ 
	(here we again use the fact that $\hat c_1(\VV_{\rho})^2=0$).
	More precisely,  we can find a finite surjective map $g: \tilde C\to C$ from a smooth curve $\tilde C$.
	Taking the base change $f_X : \tilde X \to \tilde C$ of $f$ by $g$ we get a finite surjective map $g_X : \tilde X \to X$ and then $f_X^*(g^*\VV_{\tau}) \simeq \Sym ^2 (g_X^*\VV_{\rho})$,
	where $g^*\VV_{\tau}$ is a local system on a usual curve. This implies that  $g_X^* \hat c_2(\VV _{\rho}) = \hat c_2 (g_X^*\VV _{\rho })=0\in 
	H ^{3}({\tilde X}^{\an}, \CC/\QQ (2))$, which gives the required equality in $H ^{3}(X^{\an}, \CC/\QQ (2))$.
	The proof for $c_2^{\an}$ is the same.
	The last assertion follows from the first part and \cite[Theorem 1]{DHZ}.
\end{proof}

\medskip

\begin{remark}
In case of $\hat c_i$ the above theorem is a special case of \cite[Theorem 6.1]{Re2} but with a different proof.
\end{remark}

\subsection{Chern classes of rank $2$ flat vector bundles in the Chow ring}

Let us recall that the Bloch--Beilinson conjecture predicts that if $X$ is a smooth quasi-projective variety defined over a number field then the cycle class maps  $\CH ^i (X ) _{\QQ}\to H^{2i}_{\cD} (X, \QQ (i))$ are injective for all $i\ge 1$. This conjecture can be naturally extended to smooth algebraic Deligne--Mumford stacks defined over a number field (and one can see that it is equivalent to the original conjecture).

Let $\cM$ be a polydisk Shimura modular stack as defined in \cite[Section 9]{CS}. We consider it as a smooth algebraic Deligne--Mumford stack (see ibid.). By construction, this stack is defined over a number field, so a special case of 
the Bloch--Beilinson conjecture leads to the following:

\begin{conjecture}\label{Shimura-conj}
	The cycle class map $\CH ^i (\cM ) _{\QQ}\to H^{2i}_{\cD} (\cM, \QQ (i))$ is injective for all $i\ge 1$.
\end{conjecture} 

The above conjecture implies that Chern classes of tautological  (``automorphic'') rank $2$ flat vector bundles on $\cM$ are torsion (see the proof of Proposition \ref{rk2-conj-torsion} and this is the main reason why we need it. Although Chern classes of automorphic vector bundles on  Shimura varieties have been studied (see \cite{EH1}
and \cite{EH2}), one knows only special results in this direction and the Bloch--Beilinson conjecture for Shimura varieties is still wide open.

\begin{proposition} \label{rk2-conj-torsion}
	Let  $\rho: \pi_1 (X^{\an},x)\to \GL(2, \CC)$ be a Lie irreducible representation with quasi-unipotent monodromy at infinity. If Conjecture \ref{Shimura-conj} holds then  $\Delta (V _{\rho})$ is a torsion class in  $\CH ^2 (X)$.
\end{proposition}

\begin{proof} As in the proof of Theorem \ref{Reznikov-analogue-rk2} we can reduce to the case  
	$\rho: \pi_1 (X^{\an},x)\to \SL (2, \CC)$. The only difference is that in the Chow group the self-intersection of the first Chern class is possibly non-zero and we need to use invariance of the discirrminant under tensoring with a line bundle.
	Now our assumptions imply that the image of $\rho$ is Zariski dense in  $\SL (2, \CC)$.  Therefore by \cite[Theorem 1]{CS} we need to consider two cases.
	
	In the first case  $\rho$ projectively factors through an orbicurve, i.e., there exists an orbicurve $C$,
	a fibration $f:X\rightarrow C$, and a commutative diagram
	$$
	\begin{array}{ccc}
		\pi _1(X,x) & \xrightarrow{\ ~ \rho \ } & \SL (2,\CC ) \\
		\downarrow & & \downarrow \\
		\pi _1^{\orb}(C,f(x)) & \xrightarrow{\ ~ ~ \ \ } & \PGL (2,\CC ).
	\end{array}
	$$
	Composing $ \pi _1^{\orb}(C,f(x))\to \PGL (2,\CC )$ with the adjoint representation $\PGL (2,\CC ) \to \SL (3,\CC)$
 by Lemma \ref{functorial-RH}	we get on the orbicurve $C$ a flat vector bundle $(W, \nabla_W)$ with regular singularities  such that $f^*W\simeq \Sym ^2 V_{\rho}$.  Therefore $\Delta (V_{\rho})=c_2 (\Sym^2 V_{\rho})=0$ in $\CH ^2(X)_{\QQ}$ (note that formally we should use the same arguments as that in proof of Theorem \ref{Reznikov-analogue-rk2} and pass to a finite covering of $C$). 
	
	The second case we need to consider is when $\rho$ is rigid and  $\rho$ does not projectively factor through an orbicurve. 
	In this case $\rho$ is of geometric origin and by \cite[Corollary 8.4]{CS} there is a family of abelian varieties over $X$ such that $\rho$ is a direct summand of the underlying complex monodromy representation.
	Then we would like to use \cite[Theorem 1.1]{EV}. Unfortunately, we cannot use it directly as the theorem concerns 
	weight one (polarizable) $\ZZ$-variations of Hodge structure (VHS) and we have only a weight one (polarizable) $\CC$-VHS.
	However, by \cite[Theorem 11.2]{CS} there exists a map  $f: X\to \cM$ to certain polydisk Shimura modular stack $\cM$ and a tautological rank $2$ local system $\WW$ on $\cM$ such that $\rho$ is conjugate to the monodromy representation of $f^*\WW$.
	By Theorem \ref{Reznikov-analogue-rk2} for $i\ge 2$  Deligne's Chern classes $c_i^{D}$ of the flat algebraic vector bundle associated to $\WW$ vanish in $ H^{2i}_{\cD} (\cM, \QQ (i))$. So Conjecture \ref{Shimura-conj} implies that they vanish also in  $\CH ^i (\cM ) _{\QQ}$. Now naturality of Chern classes implies that $\Delta (V_{\rho})$ vanishes in  $\CH ^2 (X)$.	
\end{proof}

Note that since in the above proposition $V$ has rank $2$, all Chern classes $c_m(V)$  vanish in $\CH ^m (X)$ for $m>2$.

\begin{lemma} \label{Marc-Levine}
	Let $X$ be a smooth complex projective variety  and let $T$ be a smooth complex affine curve.
	Let $C$ be a codimension $i$ cycle in $X\times T$. Let us assume that the class of $C_t$ is torsion in $\CH^i (X)$ for all closed points $t$ of $T$. Then there exists a uniform bound on the order of torsion of  $C_t$ that is independent of the choice of $t$.
\end{lemma}

\begin{proof}
	Everything is defined over an algebraically closed field $k_0$ of finite transcendence dimension over $\bar{\QQ}$, say by $X_0$, $T_0$ and $C_0$. Then one can find a $\CC$-point $t$ of $T$ that is a geometric generic point of $T_0$ over $k_0$.  This means concretely that $k_0(T_0)$ is a subfield of $k_0(t)$ and the field extension $k_0(T_0)\to  k_0(t)$ is necessarily finite. By assumption, the cycle $C_t$ is torsion, say $N$-torsion. The varieties one needs to realize this are also defined over a field  $L$ that is finitely generated over $k_0$. View $L$ as the field of functions of a   
	$k_0$-variety $Y$. Over a dense open subset $U$ of $Y$ one can spread out the varieties defined over $L$. Since $k_0$ is algebraically closed, one can then restrict back. This shows that the cycle $C_{0t}$ is $N$-torsion in $\CH ^i(X_0\times_{k_0}k_0(t))$. Since $t$ is a geometric generic point over $k_0$, this again can be spread out over a dense open subset $U_0$ of $T_0$ to a family of curves  and functions that exhibit $C_0\cap X_0 \times U_0$ as $N’$-torsion, $N’=N.d$, where $d$ is the (finite) field extension degree of $k_0(t)$ over $k_0(T_0)$. But then $C \cap X \times U$ is $N$-torsion, where $U$ is $U_0$ base-extended to $\CC$. If $t$ is a closed point of $T$, then there is a $0$-cycle $z$ on $U$ with the same class as $t$ in $\CH_0(T)$, so $C_z$ is equivalent to $C_t$ in $\CH ^i(X)$. Since $C_z$ is $N'$-torsion, so is $C_t$.
\end{proof}

\medskip

Proposition \ref{rk2-conj-torsion} and Lemma \ref{Marc-Levine} strongly suggest that one should have a uniform bound on the order of torsion of  $\Delta (V _{\rho})$ in families of flat vector bundles. We prove theorem showing this kind of boundedness below. Note however that algebraic monodromy groups are not well behaved in families (in general they ae not even algebraically constructible; see Subsection \ref{variation}), so we cannot directly use the above results.

\medskip

\begin{definition}
	We say that $(V, \nabla)$ is a \emph{good family of vector bundles with integrable connection parametrized by $S$} if $(V, \nabla)$ is a  vector bundle with a relative integrable connection on $X\times S /S$ and there exists  a Zariski open subset $U\subset S$ such that for all $s\in U$
	\begin{enumerate}
		\item  $(V_s, \nabla _s)= (V, \nabla)|_{X\times \{s\}}$  have regular singularities,
		\item  the monodromy representation of $(V _s, \nabla_s)$ is quasi-unipotent at infinity.
	\end{enumerate}
\end{definition}

\begin{theorem}\label{rk2-families-torsion}
	Let $X$ be a smooth complex quasi-projective variety.
	Let $(V, \nabla)$ be a  good family of rank $2$ vector bundles with integrable connection parametrized by an irreducible complex variety $S$. Assume that  Conjecture \ref{Shimura-conj} holds and for some $s_0\in S (\CC)$
	the geometric monodromy group $M_{\rm geom} (V_{s_0}, \nabla _{s_0})$ contains $\SL(2, \CC)$. Then there exists a positive integer $N$ such that for all $s\in S$ we have 
	$$N\cdot \Delta (V_s)=0 $$ 
	in $\CH ^2 (X\otimes _{\CC} k(s))$.	
\end{theorem}

\begin{proof}
Possibly passing to the desingularization of $S$ and pulling-back $(V, \nabla)$, we can assume that $S$ is smooth.
For any point $s\in S$ we denote by $PM_{\alg} (V _s, \nabla _s, x)$ the image of 
$M_{\alg} (V, \nabla, x) \subset \GL (2 , k(s))\to \PGL (2, k(s))$. Let $\eta$ be the generic point of $S$ and let us set $K=k(\eta)$.  
By Lemma \ref{surjectivity-after-extension} there exists a finite extension $K\subset L$ contained in $\bar K= k(\bar \eta)$ such that 
$$M_{\alg}(V\otimes _K \bar K, \nabla \otimes _K\bar K, x\otimes _K\bar K ) \to M _{\alg}(V\otimes _K L, \nabla \otimes _K L, x\otimes _K L )\otimes _L\bar K$$
is an isomorphism. Let $S'\to S$ be a (proper) alteration from smooth $S'$ inducing the field extension $K\subset L$.
Let $(V', \nabla ')$ be the pull-back $(V, \nabla)$ to $X\times _{\CC} S'$ and let $s_0'\in S'(\CC)$ be a fixed point lying over $s_0$. Clearly, $M_{\alg} (V'_{s_0'}, \nabla ' _{s_0'})=M_{\alg} (V_{s_0}, \nabla _{s_0})$. Therefore   
$PM_{\alg} (V'_{s_0'}, \nabla ' _{s_0'}, x) =\PGL (2, \CC)$ is $2$-dimensional over $\CC$.
 Let $\eta '$ be the generic point of $S'$. By construction we have $k(\eta ' )=L$.
By Theorem \ref{Gabber} and successive specialization we see that $PM_{\alg} (V' _{\eta '}, \nabla' _{\eta'}, x)$ has dimension at least $2$ over $L$. But by the above isomorphism, we have   
$$PM_{\alg} (V' _{\bar \eta '}, \nabla' _{\bar \eta'}, x) \simeq PM_{\alg} (V' _{\eta '}, \nabla' _{\eta'}, x)\otimes _L \bar K,$$
so  $PM_{\alg} (V '_{\bar \eta '}, \nabla' _{\bar \eta'}, x)=\PGL (2, k(\bar \eta '))$. 
This can be rewritten as $PM_{\alg} (V_{\bar \eta }, \nabla _{\bar \eta}, x)=\PGL (2, k(\bar \eta ))$.

Now our assumptions and Proposition \ref{rk2-conj-torsion} imply that $\Delta (V_{\bar \eta})$ is torsion in $\CH ^2 (X\otimes _\CC k(\bar \eta))$. Let us take positive integer $N$ such that the codimension $2$ cycle $N\cdot \Delta (V_{\bar \eta})$
on $X\otimes _\CC k(\bar \eta)$ is rationally equivalent to zero.
Let us consider the codimension $2$ cycle  $N\cdot \Delta (V)$ in $\CH^2 (X\times S)$. Since by \cite[Lemma 1A.1]{Bl}
$$\CH^2 (X _K)\simeq \varinjlim _{U\subset S \, \, \rm{open}} \CH^2 (X \times U),$$
there exists a nonempty Zariski open $U\subset S$ such that the restriction of $N\cdot \Delta (V)$ vanishes in $\CH^2 (X \times U)$.

Let us fix a point $s\in S$. We can find a smooth variety $T$, a morphism $g: T\to S$ and a point $t\in T$ mapping to $s$ such that 
$k(s)\to k(t)$ is an isomorphism, $T$ has dimension dimension  $1$ higher than the closure of $s$ and the image of $g$ intersects $U$. Let  $(V', \nabla _{V'})= g^*(V, \nabla)$, $U'= g^{-1} (U)$ and let
$Z$ be the  closure  of $t$ in $T$.
After removing a few divisors from $T$ we can assume that $Z$ is the only codimension $1$ component of $T-(U'\cap T) $. Let $i$ denote the embedding $X\times Z\hookrightarrow X\times T$ and $j$ the embedding $X \times (U'\cap T)\hookrightarrow X\times T$. We have an exact sequence 
$$\CH^1 (X\times Z)\mathop{\longrightarrow}^{i_*}  \CH^2 (X\times T )\mathop{\longrightarrow}^{j^*}  \CH^2 (X \times (U'\cap T))\longrightarrow 0,$$  
where $N\cdot  \Delta (V'|_{X\times T})$ is mapped by $j^*$ to zero. Since $i^*i_*$ is the product with $c_1  (\cO _{X\times T} (X\times Z))$, we see that after restricting $i^*(N\cdot  \Delta (V'))$ to $X\times t$
we get $0$. So  $N\cdot  \Delta (V'_t)=0$ in $\CH ^2 (X\otimes _{\CC} k(t))$. 
  Since this is the same as  $N\cdot  \Delta (V_s)=0$, we get our claim.
\end{proof}

\begin{example}\label{counterexample}
	Here we show a smooth complex projective surface and a rank $2$ unitary representation of $\pi _1 (X^{\an},x)$ such that 
	$\Delta (V)$ is not torsion in $\CH^2 (X)$ for the associated algebraic flat vector bundle $(V, \nabla)$. This gives also an example of a slope semistable vector bundle $E$  with  $\int _X \Delta (E)=0$ for which no multiple of $\Delta (E)$ is a class of an effective $0$-cycle.
	
	Let us start with a general set up. Let $X$ be a smooth complex projective variety of dimension $\ge 2$ and assume there exists line bundles $L_1, L_2\in \Pic ^0 (X)$ such that the self-intersection of $(L_1-L_2)$ is non-zero in $\CH^2 (X) _{\QQ }$. Then $E=L_1\oplus L_2$ is slope semistable with respect to any ample polarization. In fact,  $E$ corresponds to a direct sum of rank $1$ unitary representations of the fundamental group $\pi _1 (X,x)$ (here it can be seen directly as every line bundle in $\Pic^0 (X)$ admits an algebraic connection). Since $\Delta (E)=-(L_1-L_2)^2$, we get a rank $2$ unitary representation $\rho: \pi _1 (X^{\an},x)\to \GL(2, \CC )$ such that $\Delta (V_{\rho})=\Delta (E)$ is not torsion in $\CH^2 (X)$. 
	 
	Now let us show explicit examples of such varieties. Let $X$ be the Fano surface of lines in a smooth cubic threefold in $\PP^4 _{\CC}$. It is 	a smooth projective surface of general type with $p_g (X)>0$. The intersection of divisors gives a pairing
	$$\Pic ^0(X) \otimes _{\ZZ} \Pic ^0(X) \to T(X)= \ker ({\CH^2(X)}_0\to \Alb (X)),$$
	where $\CH^2(X)_0$ is the subgroup of $\CH^2(X)$ corresponding to cycles  algebraically equivalent to $0$.
	In this example the above map is surjective (see \cite[Lecture 1, Example 1.7]{Bl})	and $T(X)$ is a non-zero torsion free, divisible group (see  \cite[Lecture 1, Lemma 1.4]{Bl}, \cite[Lecture 1, Theorem 1A.6]{Bl} and \cite[p.~255]{Ja1}). Therefore there exists $L\in \Pic ^0(X)$ such that $L^2\ne 0$ is non-zero in $T(X)$. Taking $L_1=L$ and $L_2=L^{-1}$ we see that $\Delta (E)$ is a degree $0$ class, which is not torsion in $\CH^2(X)$.

	Another example of a very similar type can be recovered from \cite{Iy}. Namely, we can take as $X$ a complex abelian surface. By \cite[Theorem 3.9]{Iy} for a very general line bundle $L\in \Pic ^0(X)$ the self-intersection $L^2$ is a non-zero element of the arithmetic Deligne cohomology $H^4_{\cA \cD}(X, \QQ (2))$. In particular, $L^2$ does not vanish in $\CH ^2(X)_{\QQ}$ and as before we can take $E=L\oplus L^{-1}$. In fact, \cite[Theorems 1.1 and 1.2]{Iy} give many examples of a similar type (but in all these examples the vector bundle $E$ is reducible, so in particular it is not stable).
\end{example}

\medskip

\begin{remark}
	Let $X$ be a smooth projective surface and let $H$ be some ample divisor on $X$.
	Assume that $E$ is a slope $H$-stable vector bundle on $X$. One can ask if $\Delta( E)$ is 
	a class of an effective $0$-cycle. If $\int_X \Delta (E)=0 $ and there exists 
	a generically finite proper covering  $f: Y\to X$ such that $f^*E$ is not slope $f^*H$-stable then the above example shows that we cannot expect that $\Delta (E)$ is torsion in $\CH^2 (X)$. 
\end{remark}

\begin{remark}
	The above example shows that Theorem \ref{rk2-families-torsion} is surprising even if $X$ is projective. Namely,
let us recall that flat vector bundles with monodromy group containing $\SL(2, \CC)$ form an analytically open subset in $M_{DR} (X, 2)$. This follows by applying Proposition \ref{density} to $G=\PGL(2)$ and then applying the Riemann--Hilbert correspondence (\cite[Proposition 7.8]{Si}).
Theorem \ref{rk2-families-torsion} shows that the irreducible components of the de Rham moduli space $M_{DR} (X, 2)$ that contain flat vector bundles with monodromy group containing $\SL(2, \CC)$ do not contain points that correspond to flat vector bundles, whose underlying vector bundle is an extension of $L$ by $L^{-1}$ for some $L\in  \Pic^0 (X)$ such that $L^2\ne 0$ in $\CH^2 (X)$. 
\end{remark}

Note that Lemma \ref{Lie-irred-lemma} and	 Proposition \ref{rk2-conj-torsion} suggest the following conjecture:

\begin{conjecture}\label{too-optimistic-MIC}
	Let $X$ be a smooth complex quasi-projective variety.
	Let  $\rho: \pi_1 (X^{\an},x)\to \GL(r, \CC)$ be a Lie irreducible representation with quasi-unipotent monodromy at infinity. Then for all $i\ge 2$ we have 
	$r^i c_i(V_{\rho}) = \binom{r}{i} c_1(V_{\rho})^i$ in $\CH ^i (X)_{\QQ}$. In particular, if $\rho (\pi_1 (X^{\an},x))\subset \SL(r, \CC)$
	then $c_i(V_{\rho})$ is torsion in $\CH^i (X)$ for all $i\ge 2$. 
\end{conjecture}

It is easy to see that the second part of the conjecture is in fact equivalent to the first one. In general, it seems that this conjecture is too optimistic, but we give some partial results in the higher rank in the next subsection. For this reason we formulate it only for a single flat vector bundle instead of a family of such bundles as one could conjecture in view of Theorem \ref{rk2-families-torsion}.

\subsection{Variation of monodromy groups} \label{variation}
	
In this subsection we study variation of geometric monodromy groups	in families of rank $2$ flat vector bundles.
A standard example showing that in general we cannot expect good behaviour is the following (we recall it after \cite[Remark 2.4.4]{Ka}). Let $X=\CC ^*$  and $S=\AA ^1_{\CC}$. On $X\times S$ with coordinates $(x,t)$ we consider the trivial line bundle $V=\cO_{X\times S}$ with the relative connection $\nabla$ given by $\nabla (xd/dx) =xd/dx +t$. For any $t\not \in \QQ$ the algebraic monodromy group is equal to $\GG_m$ and for $t=p/q\in \QQ$ with $(p,q)=1$ it is equal to $\mu _q$. Note that in this example projectivization of the monodromy group is better behaved (which is not  surprising  as it is trivial in the rank one case). In view of Proposition \ref{density} under some assumptions one should get a  better behaviour that we analyze in the rank $2$ case.
More precisely, we prove the following result:

\begin{proposition}
	Let $X$ be a smooth complex quasi-projective variety. Let $(V, \nabla)$ be a  good family of rank $2$ vector bundles with integrable connection parametrized by an irreducible complex variety $S$. Assume that for some $s_0\in S (\CC)$ the geometric monodromy group $M_{\rm geom} (V_{s_0}, \nabla _{s_0})$ contains $\SL(2, \CC)$. Then there exists an open analytic subset $W\subset S^{\an}$ containing $s_0$ such that the following conditions are satisfied:
	\begin{enumerate}
		\item for all $s\in \overline W\subset S^{\an}$
		the projectivization of the monodromy of $(V_s, \nabla _s)$  at infinity lie in the same closures of conjugacy classes of matrices in $\PGL(2, \CC)$,
		\item for all $s\in \overline W\subset S^{\an}$ the monodromy representation of $(V_s, \nabla _s)$ factors through a hyperbolic Deligne--Mumford curve,
		\item  for all $s\in W$ the geometric monodromy group of $(V_s, \nabla _s)$ contains $\SL(2, \CC)$. 
	\end{enumerate} 
\end{proposition}

\begin{proof}

The following lemma is an analogue of \cite[Corollary 6.7]{CS} in case of projective representations. Since a projective representation $\rho: \pi_1 (X, x)\to \PGL(2, \CC)$ does not need to lift to $\pi_1 (X, x)\to \SL(2, \CC)$  (or even to $\pi_1 (X, x)\to \GL(2, \CC)$), we need to check that a similar proof works also in this case.

\begin{lemma} \label{closedness-of-orbicurve-factorization}
	There exists a Zariski open subset $U\subset R(X, x, \PGL (2))$ such that $\rho: \pi_1 (X, x)\to \PGL(2, \CC)$ does not factor through a hyperbolic orbicurve if and only if $\rho \in U (\CC)$.
\end{lemma}

\begin{proof}
	It is sufficient to show that the set of representations $\rho \in R(X, x, \PGL (2)) (\CC)$ such that $\rho$ factors through a hyperbolic orbicurve is closed. But by \cite[Proposition 2.8]{CS} there exists only a finite number of isomorphism classes of maps from $X$ to a hyperbolic orbicurve. So the required assertion follows from the fact that the set of representations  $\rho \in R(X, x, \PGL (2)) (\CC)$	which factor a given map to an orbicurve is closed.
\end{proof}

The following lemma is an analogue of \cite[Lemma 3.4]{CS} in case of projective representations.

\begin{lemma} \label{hyperbolic-factorization}
	If $\rho: \pi_1 (X, x)\to \PGL(2, \CC)$ is Zariski dense and it factors through an orbicurve $f: X\to C$ then $C$ is hyperbolic.
\end{lemma}

\begin{proof}
	Spherical and elliptic orbicurves have virtually abelian fundamental groups (i.e., they have an abelian subgroup of finite index; in case of spherical orbicurves this subgroup is in fact trivial). But such groups have no Zariski dense representations in  $ \PGL(2, \CC)$ and by assumption $\rho$ factors through the fundamental group of $C$, so $C$ must be hyperbolic.
\end{proof}

Let us consider the canonical map $\SL(2, \CC) \to \PGL(2, \CC)$. Each element $A\in \PGL(2, \CC)$ can be lifted to $\bar A\in \SL(2, \CC)$, which is well defined up to multiplication by $-1$. Hence $f(A):=(\Tr \bar A)^2$ is a well defined invariant of $A$.
Note that $A$ is quasi-unipotent if and only if we can choose $\bar A$ to be conjugate to one of the following matrices:
\begin{enumerate}
	\item $
	\left(\begin{matrix}
		1&0\\
		0&1\\
	\end{matrix}\right)
	$,
	\item $
	\left(\begin{matrix}
		1&1\\
		0&1\\
	\end{matrix}\right)
	$,
	\item $
	\left(\begin{matrix}
		\alpha&0\\
		0&\alpha^{-1}\\
	\end{matrix}\right)
	$, where $\alpha \ne \pm 1$ and $\alpha ^n=1$ for some positive integer $n$.
\end{enumerate}
We call closures of conjugacy classes of elements of type 2 and 3 ``large''.
One can easily see that the closure of the conjugacy class of the second matrix contains the first one. Moreover, an elementary calculation shows that $f(A)$ takes different values on large closures of conjugacy classes.

Let us fix a smooth projective compactification $\bar X$ of $X$ such that $D=\bar X-X$ is a normal crossing divisor. For each irreducible component $D_i$ of $D$ let us choose $\gamma _i\in \pi_1 (X^{\an},x)$ represented by a simple loop around that component.  Let $\rho _0:  \pi_1 (X, x)\to \PGL(2, \CC)$ denote the composition of the monodromy representation of 
$ (V_{s_0}, \nabla _{s_0})$ with the canonical quotient $\GL(2, \CC) \to \PGL(2, \CC)$. Let $C_i$ denote the largest closure of the conjugacy class of a quasi-unipotent matrix that contains $\rho _{0} (\gamma _i)$ (since  $\rho _{0} (\gamma _i)$ is quasi-unipotent, by the above it is well-defined).

As in the proof of Corollary \ref{moduli-map-for-de_Rham-families} we can find a Zariski open neighbourhood  $U_0\subset S$ of $s_0$
and an analytic morphism  $\varphi _0 : U_0^{\an}\to R(X,x, \PGL(2))^{\an}$ such that for every $s\in U_0(\CC)$ the point $\varphi _{0} (s)$ corresponds to the projectivization of the monodromy representation $\rho _s$ of $(V_s, \nabla _s)$.

We claim that $\varphi _0(U_0^{\an})\subset  R(X, x, \PGL(2) , \{C_i\})^{\an}$. 
To see that consider $f_i: R(X, x, \PGL(2))^{\an} \to \CC$ defined by $f_i (\rho)= f(\rho (\gamma _i))$. 
By assumption for every $s \in U_0^{\an}$ the elements $(\varphi _0 (s)) (\gamma _i)$ are quasi-unipotent. Since $f_i\circ \varphi_0$
is continuous, $U_0 ^{\an}$ is pathwise-connected and images by $f$ of two distinct large conjugacy classes are not path-connected by images of quasi-unipotent matrices, we see that $f_i\circ \varphi_0 (U_0 ^{\an})$ is a point for each $i$. This implies our claim and it gives an induced map $U_0^{\an}\to R(X,x, \PGL(2),  \{C_i\})$ that we also denote by $\varphi _0$.

By Proposition \ref{density} there exists an open subset $U_1\subset  R(X,x, \PGL(2),  \{C_i\})$ corresponding to representations, whose image is Zariski dense in $\PGL (2, \CC)$. By assumption we have $\varphi _0 (s _0) \in U_1(\CC)$.
Let us set $U_2= U\cap  R(X,x, \PGL(2),  \{C_i\})$, where $U$ is as in Lemma \ref{closedness-of-orbicurve-factorization}.

Assume that  $\rho \in U_1(\CC)\cap U_2(\CC)$. Since $\rho $ is Zariski dense, by Lemma \ref{hyperbolic-factorization} it does not factor through a map of map from $X$ to an orbicurve. But by assumption $\rho $ is not rigid, so we get 
a contradiction with \cite[Theorem 6.8]{CS}.  Therefore $U_1\cap U_2=\emptyset$.

Let $T$ be an irreducible component of $R(X,x, \PGL(2),  \{C_i\})$ containing $\varphi _0 (s _0)$.
Since $\varphi _0 (s _0) \in U_1(\CC)$ and $T$ is irreducible we must have $T\cap U_2=\emptyset$. It follows that every $\CC$-point of $T$ corresponds to a representation that factors through a hyperbolic orbicurve
and as $W$ we can take $\varphi _0^{-1} ((T\cap U_1)^{\an})$.
\end{proof}

\begin{remark}
Note that although $\overline W$ corresponds to an irreducible component of $R(X,x, \PGL(2),  \{C_i\})$, we cannot claim that the assertions of the proposition hold for all closed points of $S$. If this were true Theorem \ref{rk2-families-torsion} would follow from Proposition \ref{rk2-conj-torsion} and Lemma \ref{Marc-Levine}.
\end{remark}

\subsection{Chern classes of higher rank flat vector bundles in the Chow ring} \label{higher-rank-case}

Let $E$ be a rank $r>0$ coherent sheaf on a smooth algebraic variety $X/k$.
We can write
$$\log (\ch (E))=\log r+\sum _{i\ge 1} (-1)^{i+1}\frac{1}{i!r^i}\Delta_i(E)$$
for some classes $\Delta_i(E)\in \CH^{i}(X) _{\QQ}$. In fact, these classes are 
polynomials in Chern classes of $E$ with integral coefficients, so we can consider
them in $\CH^i (X)$. These classes are called \emph{higher discriminants} of $E$
and to the author's knowledge they were first considered by Drezet in \cite{Dr}, although with a different normalization with rational coefficients.
It is easy to see that for $i\ge 2$ these classes are invariant under tensoring with a line bundle, i.e.,
for any line bundle $L$ we have $\Delta_i(E\otimes L)=\Delta_i(E)$.  This follows
immediately from the fact that
$$\log (\ch (E\otimes L))= \log (\ch (E) \cdot \ch(L)) =\log (\ch (E))+c_1(L).$$

\medskip

Let us recall that a linear algebraic group is called \emph{almost simple} if all its proper normal Zariski closed subgroups are finite. Such group is necessarily connected and semi-simple.

\begin{proposition} \label{Zuo}
	Let $X$ be a smooth complex projective variety. Let  $\rho: \pi_1 (X^{\an},x)\to \GL(r, \CC)$ be a representation, whose 
	geometric monodromy group $G= M_{\rm geom}(\rho)$ is connected and almost simple. If $\rho$ is not properly rigid then for all $i> \rk G$ we have 
	$c_i(V_{\rho}) = 0$ 
	in $\CH ^i (X)_{\QQ}$. Moreover, all combinations of Chern classes of $V_{\rho}$ that live in 	 
	$\CH ^i (X)_{\QQ}$ vanish for $i> \rk G$.
\end{proposition}

\begin{proof}
	By \cite[Theorem 3]{Zu2} there exists an alteration $g: \tilde X\to X$ and a fibration $f: \tilde X\to Y$ such that $g^*\rho$ factors through $\tau: \pi_1 (Y^{\an},y)\to G$ and $\dim Y \le \rk G$. 
	By Lemma \ref{functorial-RH} we have in $\CH ^i (\tilde X)_{\QQ}$
	$$ g^*c_i (V_{\rho}) =c_i (V_{g^*\rho})= f^*c_i (V_{\tau})=0$$
	for $i> \dim Y$. This implies the required vanishings in $\CH ^i (X)_{\QQ}$. The last assertion follows by the same arguments.
\end{proof}

\begin{definition}
	We say that  a representation  $\rho: \pi_1 (X^{\an},x)\to \GL(r, \CC)$ is \emph{projectively rigid} if the 
 $\pi_1 (X^{\an},x)\to \PGL(r, \CC)$ induced from $\rho$  by  the group extension $\GL(r, \CC)\to \PGL(r, \CC)$ is rigid.
\end{definition}

\begin{proposition}\label{rk-r-one}
Let $X$ be a smooth complex projective variety. Let  $\rho: \pi_1 (X^{\an},x)\to \GL(r, \CC)$ be a representation, whose geometric monodromy group $M_{\rm geom}(\rho)$ contains $\SL (r, \CC)$. If $\rho$ is not projectively rigid then 
$\Delta _r(V_{\rho})$ is torsion in $\CH ^r (X)$.	
\end{proposition}

\begin{proof}
We consider the short exact sequence of groups
$$1\to \mu _r\to \CC^*\to \CC^*\to 1,$$
where the map $\CC^*\to \CC^*$ is given by $z\to z^r$. \cite[Lemma 2.6]{LS} shows that there exists a finite surjective morphism $\tilde X\to X$ of algebraic varieties 	and a rank $1$ representation $\tau : \pi_1 (\tilde X, \tilde x)\to \CC^*$ such that $f^*(\det \rho )=\tau ^r$.
Then the image of  $\tilde \rho:= f^*\rho\otimes \tau ^{-1}: \pi_1 (\tilde X, \tilde x) \to \GL(r,\CC)$ is contained in $ \SL (r, \CC)$.  Lemma \ref{functorial-RH} implies that $V_{\tilde \rho} \simeq f^* V_{\rho} \otimes V_{\tau} ^{-1}$, so 
$$\Delta _i (V_{\tilde \rho})= \Delta _i(f^*V_{\rho} )= f^* \Delta _i(V_{\rho} )$$
for $i\ge 2$. Since $f^*: \CH ^i (X) _{\QQ}\to \CH ^i (\tilde X) _{\QQ}$ is injective,  it is sufficient to prove that $\Delta _r(V_{\tilde \rho})$ is torsion in $\CH ^r (X)$. By Lemma \ref{connected-components-monodromy} 
$\tilde \rho: \pi_1 (\tilde X, \tilde x)\to
\SL (r, \CC)$ is Zariski dense. Since $\rho$ not projectively rigid, the induced representation
 $P\tilde \rho: \pi_1 (\tilde X, \tilde x)\to PGL (r, \CC)$ is not rigid. Since $\Delta _r(V_{\tilde \rho})$ 
is a characteristic class of  the $\PGL(r, \CC)$-bundle associated to $P\tilde \rho$,  the required assertion
follows from Proposition \ref{Zuo}.
\end{proof}

In the rank $3$ case we can also deal with rigid representations:

\begin{proposition} \label{rk3-conj-torsion}
	Assume that $X$ is projective. Let  $\rho: \pi_1 (X^{\an},x)\to \SL(3, \CC)$ be a rigid integral irreducible representation.  If Conjecture \ref{Shimura-conj} holds then for all $i\ge 1$ we have $c_i (V_{\rho})=0$ in  $\CH ^i (X) _{\QQ}$.
\end{proposition}

\begin{proof}
	The proof is the same as that of Proposition \ref{rk2-conj-torsion}, except that instead of using the results
	of \cite{CS} we need to use \cite{LS}.
\end{proof}

\begin{theorem}\label{rk-r-families-torsion}
	Let $X$ be a smooth complex projective variety and let $S$ be an  irreducible complex variety.
	Let $(V, \nabla)$ be a  rank $r$ vector bundles with a relative integrable connection on $X\times S/ S$. 
	Assume that $(V_{\bar \eta}, \nabla_{\bar \eta})$, where $\bar \eta$ is the generic geometric point of $S$, is not projectively rigid and for some $s_0\in S (\CC)$	the geometric monodromy group $M_{\rm geom} (V_{s_0}, \nabla _{s_0})$ contains $\SL(n, \CC)$. Let us fix a collection of integers $\{i_j\}$ such that $i_j\ge 2$ and
	$\sum i_j=r$. Then there exists a positive integer $N$ such that for all $s\in S$ we have 
	$$N\cdot \prod _{j}\Delta _{i_j} (V_s)=0 $$ 
	in $\CH ^r (X\otimes _{\CC} k(s))$.	
\end{theorem}

\begin{proof}
	The proof is the same as that of Theorem \ref{rk2-families-torsion} except that instead of Proposition \ref{rk2-conj-torsion} one needs to use Proposition \ref{rk-r-one}.
\end{proof}

\medskip

\section{Higgs version of Bloch's conjecture}

Let us recall that Simpson  in \cite{Si0} proved that on a smooth complex projective variety there exists a  one-to-one correspondence  between irreducible flat vector bundles and stable Higgs bundles with Chern classes vanishing in $H^{2m} (X, \QQ )$ for $m\ge 1$. This suggests that one can consider an analogue of Bloch's conjecture in the world of semistable Higgs bundles with vanishing rational Chern classes. This analogue has an advantage that the conjecture becomes a purely algebraic statement that can be formulated over an arbitrary field. Note that in principle this analogue is not equivalent to the usual Bloch's conjecture as Simpson's correspondence is not algebraic. However, we show equivalence of weak forms of these conjectures (see Proposition \ref{equivalence}). We can also consider the following conjecture that is motivated by Conjecture \ref{too-optimistic-MIC} and Lemma \ref{Lie-irred-lemma}.

\begin{conjecture}\label{too-optimistic-Higgs}
	Let $X$ be a smooth projective variety of dimension $d$ defined over an algebraically closed field $k$ and let 
	$H$ be an ample divisor on $X$. Let	$(E, \theta)$ be a slope $H$-stable Higgs vector bundle with $\int _X\Delta (E)\cdot H^{d-2}=0$.	Assume that for every smooth projective variety $Y/k$ and  a generically finite map $f: Y\to X$  
	the sheaf $f^*(E, \theta)$ is slope stable with respect to some ample divisor. Then $r^i c_i(E) = \binom{r}{i} c_1(E)^i$ in $\CH^i (X)_{\QQ}$ for every $i\ge 1$. 
\end{conjecture}

The formulation of the above conjecture takes into account the fact that the condition  $\int _X\Delta (E)\cdot H^{d-2}=0$ implies equalities  $r^i c_i(E) = \binom{r}{i} c_1(E)^i$ in the de Rham cohomology (see \cite[Theorem 2.16]{La4}). Another fact that is hidden in the formulation is that (in the characteristic zero case) a Higgs bundle with  $\int _X\Delta (E)\cdot H^{d-2}=0$ that is slope stable with respect to some ample divisor is ample with respect to all ample divisors.

Analogously to weak Bloch's conjecture one can also formulate the following Higgs version of this conjecture:

\begin{conjecture}\label{weak-Bloch-Higgs}
	Let $X$ be a smooth projective variety of dimension $d$ defined over an algebraically closed field $k$ of characteristic zero and let 
	$H$ be an ample divisor on $X$. Let	$(E, \theta)$ be a slope $H$-semistable Higgs vector bundle with $\int _X\Delta (E)\cdot H^{d-2}=0$. Then $r^i c_i(E) = \binom{r}{i} c_1(E)^i$ in $B^i (X) _{\QQ}$ for every $i\ge 1$. 
\end{conjecture}

\subsection{Comparison of weak Bloch's conjecture with its Higgs version} \label{Bloch}

In this subsection we show that for smooth complex projective varieties weak Bloch's conjecture is 
equivalent to the corresponding version for Higgs bundles.

Let $X$ be a smooth $n$-dimensional variety defined over some algebraically closed field $k$ 
of characteristic zero. The following proposition translates the weak form of Bloch's conjecture \ref{weak-Bloch} 
 into a statement about stable Higgs bundles (with trivial determinant):

\begin{proposition} \label{equivalence}
	The following conditions are equivalent: For every smooth projective variety $X$ defined over $k$ and
	\begin{enumerate}
		\item every flat vector bundle $(V, \nabla)$ on $X$ the  Chern classes  $c_m(V)$ are torsion in $B^m(X)$ for all $m\ge 1$. 
		\item  every irreducible flat vector bundle $(V, \nabla)$ with $\det V\simeq \cO_X$  the  Chern classes  $c_m(V)$ are torsion in $B^m(X)$ for all $m\ge 2$.
		\item  for some ample divisor $H$ on $X$ 
		for any slope $H$-stable Higgs bundle $(E, \theta )$ on $X$ with $\det E\simeq \cO_X$ and $\int _X c_2 (E)\cdot H^{n-2}=0$ the  Chern classes  $c_m(E)$ are torsion in $B^m(X)$ for all $m\ge 2$.
		\item for every ample divisor $H$ on $X$ 
		for any slope $H$-semistable reflexive Higgs sheaf $(E, \theta)$ on $X$ with $\int _X \Delta (E)\cdot H^{n-2}=0$ we have for all $m\ge 1$ equality 
		$$r^m c_m(E)={\binom{r}{m}} c_1(E)^m$$  in  $B^m(X)_{\QQ}$.  
		\item  for some ample divisor $H$ on $X$ 
		for any slope $H$-stable system of Hodge bundles $(E, \theta )$ on $X$ with $\det E\simeq \cO_X$ and $\int _X c_2 (E)\cdot H^{n-2}=0$ the  Chern classes  $c_m(E)$ are torsion in $B^m(X)$ for all $m\ge 2$.
		\item for every ample divisor $H$ on $X$ 
		for any slope $H$-semistable reflexive system of Hodge sheaves $(E, \theta)$ on $X$ with $\int _X \Delta (E)\cdot H^{n-2}=0$ we have for all $m\ge 1$ equality 
		$$r^m c_m(E)={\binom{r}{m}} c_1(E)^m$$  in  $B^m(X)_{\QQ}$.  
	\end{enumerate}
\end{proposition}

\begin{proof}
Clearly, $(4)$ implies $(3)$. Let us check that $(3)$ implies $(4)$.
First let us show that we can assume that $\det E\simeq \cO_X$.
By \cite[Theorem 3.2]{Si0} (or \cite[Theorem 2.13]{La4}) $E$ is locally free.
By the Bloch--Gieseker theorem (see, e.g., \cite[Theorem 4.1.10]{Laz}) we can find a smooth projective variety $Y$ and a finite morphism $f: Y\to X$ such that $f^*(\det E)\simeq L^{\otimes r}$ for some line bundle $L$. Then $\tilde E= (f^*E)\otimes L^{-1}$ with an induced Higgs field $\tilde \theta: \tilde E\to \tilde E\otimes f^*\Omega_X\to \tilde E\otimes \Omega_Y$	is a slope $f^*H$-semistable Higgs bundle with $\det \tilde E \simeq \cO_Y$.
Note that $f^*: B^m (X)_{\QQ}\to B^m (Y)_{\QQ} $ is injective. So by \cite[Lemma 1.11]{La4}
it is sufficient to show that $\Delta _m (f^*E)=0 $ for $m\ge 2$. Applying our assumption to $(\tilde E, \tilde \theta)$ this follows from $\Delta _m (f^*E) = \Delta _m (\tilde E)=0$ for $m\ge 2$. 
Now we need to show that we can reduce to the case when $(E, \theta)$ is slope $H$-stable.

So from now on we can assume that $\det E\simeq \cO_X$.
Let $E_0=0\subset (E_1, \theta_1)\subset ...\subset (E_m, \theta_m)=(E, \theta)$ be 
a Jordan--H\"older filtration  of $(E, \theta)$ and let us set $(E^i,\theta^i)=(E_{i},\theta_i)/(E_{i-1},\theta_{i-1})$.
Then in $\CH^2 (X) _{\QQ}$ we have
	$$\frac{\Delta (E)}{r}=\sum _i\frac{\Delta (E^i)}{r_i}- \frac{1}{r}\sum _{i<j} r_ir_j \left( \frac{c_1(E^i)}{r_i}- \frac{c_1(E^j)}{r_j}\right) ^2,$$
	where $r$ is the rank of $E$ and $r_i$ denotes the rank of $E^i$. 
	By the definition of a Jordan--H\"older filtration we have 
	$$\int_X\left( \frac{c_1(E^i)}{r_i}- \frac{c_1(E^j)}{r_j}\right)\cdot H^{n-1}=0.$$
	Then the Hodge index theorem implies that 
	$$\int _X\left( \frac{c_1(E^i)}{r_i}- \frac{c_1(E^j)}{r_j}\right) ^2\cdot H^{n-2}\le 0.$$
	Since by the assumption we have  $\int _X \Delta (E)\cdot H^{n-2}=0$ and  Bogomolov's inequality for slope semistable Higgs sheaves gives $\int _X \Delta (E^i)\cdot H^{n-2}\ge 0$, the above formula implies that 	$$\int _X\left( \frac{c_1(E^i)}{r_i}- \frac{c_1(E^j)}{r_j}\right) ^2\cdot H^{n-2}= 0$$
	and $\int _X \Delta (E^i)\cdot H^{n-2}= 0$ for all $i$ and $j$.
	
	By \cite[Theorem 9.6.3]{Kl} if $D$ is a Cartier divisor on $X$ then $mD$ is algebraically equivalent to zero for some positive integer $m$ if and  only if $\int _X D\cdot H^{n-1}=\int _X D^2\cdot H^{n-2}=0$. It follows that 
	some positive multiple of $\left( \frac{c_1(E^i)}{r_i}- \frac{c_1(E^j)}{r_j}\right) $ is algebraically equivalent to zero
	and hence $\frac{\Delta (E)}{r}=\sum _i\frac{\Delta  (E^i)}{r_i}$ in $B^2 (X)_{\QQ }$. Note also that since  $\det E\simeq \cO_X$, we have $c_1(E^i)=0$ in $B^1 (X)_{\QQ}$ for all $i$.
	Note that  $\int _X \Delta (E^i)\cdot H^{n-2}= 0$ for some ample $H$ implies the same equality for all ample $H$ (this follows, e.g., from Simpson's correspondence although it can also be proven directly).
	So our assumption and reduction to the trivial determinant case show that $c_m (E^i)=0$ for all $i$. This implies that $c_m ^B (E)=0$ as required.

	The same arguments give equivalences $(1)\iff (2)$ and $(5)\iff (6)$. Clearly, $(3)$ implies $(5)$. 
	To prove $(6)\Rightarrow (3)$ note that by \cite[Theorem 1.3]{La4} there exists a decreasing filtration
	$E=N^0\supset N^1\supset ... \supset N^l=0$ such that $\theta (N^i)\subset N^{i-1}\otimes \Omega_X$ and the associated graded is a slope $H$-semistable system of Hodge sheaves. Applying $(6)$ to this system, we get $c_m (E)=0$ in $B^m (X)_{\QQ}$. 
		
	The fact that $(2)$ implies $(5)$ follows from the fact that every slope $H$-stable system of Hodge bundles as in $(5)$ underlies a complex polarizable variation of Hodge structure, so also a flat vector bundle. 
	Finally $(5)$ implies $(2)$ because  by \cite[Theorem 1.4]{La4} every flat vector bundle can be deformed to a complex polarizable variation of Hodge structure  which underlies a system of Hodge bundles.
\end{proof}

Similar arguments as that in the proof of $(3)\Rightarrow (4)$ and Reznikov's variant of Bloch's conjecture (see \cite[Theorem 1.1]{Re}) give the following result.

\begin{theorem}
	Let $X$ be a  smooth complex projective variety and let $H$ be an ample divisor on $X$.  
	For any slope $H$-semistable  Higgs bundle $(E, \theta)$ on $X$ with $\int _X \Delta (E)\cdot H^{\dim X-2}=0$ we have for all $m\ge 1$ equality 
	$$r^m c_m^D(E)\otimes  \QQ={\binom{r}{m}} c_1^D(E)^m\otimes  \QQ$$  in  Deligne's cohomology $H^{2m}_{\cD} (X, \ZZ (m))\otimes \QQ$.  
\end{theorem}

\subsection{Higgs version of Bloch's conjecture in positive characteristic}\label{positive-char}

Note that we formulate Conjecture \ref{too-optimistic-Higgs} in an arbitrary characteristic. However, in positive characteristic the condition that  for all generically finite maps $f$ the pullback $f^*(E, \theta)$ is slope stable, implies that already the vector bundle $E$ is strongly slope stable, i.e., for all $n\ge 0$ the $n$-th Frobenius pullback $(F_X^n)^*E$ is slope stable. This follows easily by taking $f=F_X^n$ for $n\ge 1$ as then the Higgs field of   $f^*(E, \theta)$ vanishes.  Therefore if $k$ is an algebraic closure of a finite field then the conjecture follows from the following proposition:

\begin{proposition} \label{Bloch-finite-field}
	Let $X$ be a smooth projective variety of dimension $n$ defined over an algebraic closure of a finite field $k$. Let 
	$H$ be an ample divisor and let	$E$ be a locally free $\cO_X$-module of rank $r$ with $\int _X\Delta (E)\cdot H^{n-2}=0$. Assume that $E$ is strongly slope $H$-semistable, i.e., for  every $m$ the $m$-th Frobenius pullback $(F_X^m)^*E$ is slope $H$-semistable. Then  for all $i\ge 1$ we have $r^i c_i(E) = \binom{r}{i} c_1(E)^i$ in $\CH ^i (X)_{\QQ}$.
	In particular, $\Delta (E)$ is torsion in $\CH^2 (X)$. 
\end{proposition}

\begin{proof}
	There exists a finite subfield $l\subset k$ such that both $X$ and $E$ are defined over $l$.	
	Let us write $p^m=rs_m+q_m$ for some  $s_m$ and $0\le q_m<r$ and let us set $G_m:=(F_X^m)^*E\otimes \det E^{-s_m}$. 
	The sheaves $G_m$ are locally free $\cO_X$-modules defined over the same finite field $l$. Note that $c_1(G_m)= q_m c_1(E)$ and $\Delta (G_m)=p^{2m}\Delta (E)$. Hence $\{c_1(G_m) \}_{m\ge 0}$ can take only finitely many values and $\int _X\Delta (G_m)\cdot H^{n-2}=0$ for all $m\ge 0$. By \cite[Theorem 4.4]{La1} the family $\{G_m \}_{m\ge 0}$ is bounded and since all elements are defined over the same finite field, this family is actually finite. It follows that there exist some $m>m'$
	such that $G_m\simeq G_{m'}$. In particular, we have  $\Delta (G_m) =\Delta (G_{m'})$ and hence $(p^{2m}-p^{2m'})\Delta (E)=0$ in $\CH^2 (X)$. Similarly,  in the notation of \cite{La4}  the same argument shows that all  $\Delta _i (E)$ are torsion in $\CH^i (X)$ for $i\ge 2$. By  \cite[Lemma 1.7]{La4} this implies that for all $i\ge 1$ we have
	$r^i c_i(E) = \binom{r}{i} c_1(E)^i$ in $\CH ^i (X)_{\QQ}$.
\end{proof}

\begin{proposition} \label{Bloch-in-B-char-p}
	Let $X$ be a smooth projective variety of dimension $n$ defined over an algebraically closed field $k$ of positive characteristic. Let $H$ be an ample divisor and let	$E$ be a locally free $\cO_X$-module (of finite rank) with $\int _X\Delta (E)\cdot H^{n-2}=0$. Assume that $E$ is strongly slope $H$-semistable. Then for all $i\ge 1$ we have $r^i c_i(E) = \binom{r}{i} c_1(E)^i$ in $B^i(X)$.
\end{proposition}

\begin{proof}
	The proof is very similar to that of Proposition \ref{Bloch-finite-field}. Namely, in notaton from that proof  the family $\{G_m \}_{m\ge 0}$ is bounded. Therefore  there exist some $m>m'$ such that $G_m$ and $G_{m'}$ are elements of the same flat family of vector bundles on $X$. So  $\Delta _i(G_m) =\Delta _i (G_{m'})$ in $B^i (X)$. As before this implies equalities 	$r^i c_i(E) = \binom{r}{i} c_1(E)^i$ in $B^i(X)$.
\end{proof}

\begin{corollary}
	Let $X$ be a smooth projective variety of dimension $n$ defined over an algebraically closed field $k$ of positive characteristic. Let $H$ be an ample divisor and let	$E$ be a  strongly slope $H$-semistable rank $r$ vector  bundle.
	If   $\int _X \ch_1 (E)\cdot H^{n-1}=\int _X\ch_2 (E)\cdot H^{n-2}=0$ then $c_i (E)$ is torsion in $B^i(X)$ for all $i \ge 1$.
\end{corollary}

\begin{proof}
	By \cite[Theorem 3.2]{La1} we have  $\int _X\Delta (E)\cdot H^{n-2}\ge 0$. Hence by the Hodge index theorem
	$$0=2r \int _X \ch_2 (E)\cdot H^{n-2}= \int _X( c_1 (E)^2-\Delta (E) )\cdot H^{n-2}\le \int _X c_1 (E)^2\cdot H^{n-2}\le 
	\frac{(\int _X c_1 (E)\cdot H^{n-1} )^2}{H^n}=0.
	$$
	It follows that $\int _X c_1 (E)\cdot H^{n-1}=\int _Xc_2 (E)\cdot H^{n-2}=0$. So by \cite[Theorem 9.6.3]{Kl} the class of 
	$ c_1 (E)$ in $B^1 (X)$ is torsion. Now the required assertion follows directly from Proposition \ref{Bloch-in-B-char-p}.
\end{proof}

Finally, one can prove a variant of Conjecture \ref{weak-Bloch-Higgs} in positive characteristic:

\begin{theorem}\label{Bloch-in-char-p-Higgs}
	Let $X$ be a smooth projective variety of dimension $n$ defined over an algebraically closed field $k$ of positive characteristic $p$. Let $D$ be a simple normal crossing divisor on $X$ and let us assume that $(X, D)$ is liftable to $W_2(k)$. Let us also fix an ample divisor $H$ on $X$.
	Let $(E, \theta)$ be a slope $H$-semistable logarithmic Higgs bundle of rank $r\le p$ on $(X, D)$.
	If  $\int _X\Delta (E)\cdot H^{n-2}=0$ then for every $i\ge 1$
	we have $r^i c_i(E) = \binom{r}{i} c_1(E)^i$
	in $B^i (X)_{\QQ}$. Moreover, if $k$ is an algebraic closure of a finite field then equalities $r^i c_i(E) = \binom{r}{i} c_1(E)^i$ hold in $\CH^i (X)_{\QQ}$.
\end{theorem}

\begin{proof}
	The proof is quite similar to proofs of Propositions \ref{Bloch-finite-field} and \ref{Bloch-in-B-char-p}. It 
	essentially follows from the proof of \cite[Theorem 2.2]{La4}, so we only sketch the arguments. 
	Our assumptions imply that we can consider the canonical Higgs-de Rham sequence starting with $(E_0,\theta_0)=(E, \theta)$
	$$ \xymatrix{
		& (V_0, \nabla _0)\ar[rd]^{\Gr _{S_0}}&& (V_1, \nabla _1)\ar[rd]^{\Gr _{S_1}}&\\
		(E_0, \theta _0)\ar[ru]^{C^{-1}}&&(E_1, \theta_1)\ar[ru]^{C^{-1}}&&...\\
	}$$ 
	where $C^{-1}$ denotes the inverse Cartier transform and $(E_{m+1}, \theta _{m+1})$ is a slope $H$-semistable
	logarithmic system of Hodge sheaves associated to $(V_i, \nabla _i)$ via Simpson's
	filtration. \cite[Theorem 2.2]{La4} implies that all $E_i$ are vector bundles. 
	
	Let us write $p^m=rs_m+q_m$ for some non-negative integers $s_m$
	and $0\le q_m<r$ and set $(G_m, \theta _{G_m}):=(E_m, \tilde
	\theta _m)\otimes \det E^{-s_m}$. We have equalities
	$\Delta_i (G_m)=\Delta_i(E_m)=p^{im}\Delta _i(E)$ in $\CH^i (X)_{\QQ}$ for $1\le i\le n$,
	so  $\int _X\Delta (G_m)\cdot H^{n-2}=0$.  Note also that  $c_1(G_m)= q_m c_1(E)$ in $\CH^1 (X)_{\QQ}$ 
	can take only  finitely many values,  so \cite[Theorem 3.4]{La1} (or more precisely \cite[Theorem 1.2]{La4}) implies that the family of $H$-semistable logarithmic Higgs bundles $\{(G_m, \theta _{G_m})\} _{m\ge 0}$ is
	bounded. In particular, the set $\{\Delta_i (G_m)\}_{m\ge
		0}=\{p^{nm}\Delta _i(E)\}_{m\ge 0}$ considered in $B^i (X)_{\QQ}$ is finite. Hence $\Delta
	_i(E)=0$ in $B^i (X)_{\QQ}$ for all $i\ge 0$, which implies the required equalities by \cite[Lemma 1.11]{La4}.
	
	To see the last part of the theorem one needs to note that if $k$ is an algebraic closure of a finite field then
	all $(G_m, \theta _{G_m})$ are defined over the same finite field so the set $\{\Delta_i (G_m)\}_{m\ge
		0}=\{p^{nm}\Delta _i(E)\}_{m\ge 0}$ is actually finite in $\CH^i (X)_{\QQ}$.
\end{proof}

\begin{remark}\label{Bloch-in-char-p-MIC}
	In the above theorem we can replace the Higgs bundle $(E, \theta)$ with a slope $H$-semistable bundle with a logarithmic integrable connection. In particular, we get proof of a version of usual Bloch's conjecture
	in positive characteristic. 
\end{remark}

\subsection{Can we prove Bloch's conjecture by reduction?} \label{reductions?}

Let $X$ be a smooth complex projective variety and let $(V, \nabla)$ be a rank $r$ flat vector bundle on $X$. 
Let us recall that flat vector bundles have vanishing Chern classes so they are slope semistable with respect to an arbitrary ample divisor.

There exists a finitely generated over $\ZZ$ subring $R\subset \CC$ , a smooth projective scheme $\cX \to S=\Spec R$ and a vector bundle $(\cV, \nabla _{\cV})$ with a relative integrable connection with respect to $S$
such that $X=\cX \otimes_R {\CC}$ and   $(V, \nabla) = (\cV, \nabla _{\cV})\otimes _R \CC$. We can also fix a relative ample line bundle $\cH$.

For an open subset $U\subset S$ and all closed points $s\in U$ the reduction $\cX_s$ is liftable to $W_2 (k(s))$
and $(\cV _s, (\nabla _{\cV})_s)$ are slope $\cH _s$-semistable (by openness of slope semistability).
Since $k(s)$ is finite we can use Theorem \ref{Bloch-in-char-p-Higgs} (see Remark \ref{Bloch-in-char-p-MIC})
to get for all closed points $s\in U$ equalities  
$$r^i c_i(\cV_s) = \binom{r}{i} c_1(\cV_s)^i$$ 
in $\CH^i (\cX _s)_{\QQ}$.

This motivates the following questions:

\begin{question}
Let us fix a class  $C\in \CH ^i (\cX)$. Assume that for all closed points $s\in U$ the class $C_s$ is torsion in $\CH ^i (\cX _s)$. 
\begin{enumerate}
	\item Is it true that $C_{\CC}\otimes {\QQ}$ lies in $F^i \CH ^i (X)_{\QQ}$?
	\item Is it true that the class of  $C_{\CC}$ is torsion in $B^i (X)$?
\end{enumerate}
\end{question}

If the answer to question 1 is positive then by the above Bloch's conjecture \ref{Bloch} holds. If the answer to question 2 is positive then weak Bloch's conjecture \ref{weak-Bloch} holds. Both questions are easily seen to have positive answers in case $i=1$. Moreover, the second question has a positive answer for $i=\dim X$.

Note that the first question is somewhat vague as we do not know existence of the Bloch--Beilinson filtration. 
In general, we cannot expect that  $C_{\CC}\otimes {\QQ}$  vanishes in $\CH ^i (X)_{\QQ}$. Indeed, if $L\in \Pic ^0 (X)$ is a non-torsion line bundle then all reduction of  $L$ to positive characteristic are  torsion line bundles (as they are defined over a finite field), whereas the class of $c_1 (L)\otimes \QQ $ in  $\CH ^1 (X)_{\QQ}$ is non-zero.

\section*{Acknowledgements}
The author would like to thank  H\'el\`ene Esnault, Marc Levine and Carlos Simpson for helpful e-mail discussions.
The author was partially supported by Polish National Centre (NCN) contract numbers 
2018/29/B/ST1/01232.


\begin{thebibliography}{SGA1}


\bibitem{AB} A'Campo, Norbert; Burger, Marc 
R\'eseaux arithm\'etiques et commensurateur d'apr\'es G. A. Margulis.
\emph{Invent. Math.} {\bf 116} (1994), 1--25. 

\bibitem {Bl}  Bloch, Spencer Lectures on algebraic cycles. Second edition. \emph{New Mathematical Monographs} {\bf 16}. Cambridge University Press, Cambridge, 2010. xxiv+130 pp.

\bibitem{Bl2} Bloch, Spencer  Applications of the dilogarithm function in algebraic K-theory and algebraic geometry. \emph{Proceedings of the International Symposium on Algebraic Geometry (Kyoto Univ., Kyoto, 1977)}, pp. 103--114, Kinokuniya Book Store, Tokyo, 1978.

\bibitem{CE}  Corlette, Kevin; Esnault, H\'el\`ene Classes of local systems of $\QQ$-hermitian vector spaces, Preprint 1994,
appendix to  Bismut, J. M. Eta invariants, differential characters and flat vector bundles. \emph{Chinese Ann. Math. Ser. B} {\bf 26} (2005), 15--44. 

\bibitem {CS} Corlette, Kevin; Simpson, Carlos
On the classification of rank-two representations of quasiprojective fundamental groups. 
\emph {Compos. Math.} {\bf 144} (2008), 1271--1331. 





\bibitem {DMOS} Deligne, Pierre; Milne, James S.; Ogus, Arthur; Shih, Kuang-yen
Hodge cycles, motives, and Shimura varieties.
\emph{Lecture Notes in Mathematics} {\bf 900}. Springer-Verlag, Berlin-New York, 1982. ii+414 pp.

\bibitem {Dr}  Drezet, Jean-Marc 
Exceptional bundles and moduli spaces of stable sheaves on $\PP^n$. Vector bundles in algebraic geometry (Durham, 1993), 101--117, \emph{London Math. Soc. Lecture Note Ser.} {\bf 208}, Cambridge Univ. Press, Cambridge, 1995. 



\bibitem{DHZ}  Dupont, Johan; Hain, Richard; Zucker, Steven Regulators and characteristic classes of flat bundles. The arithmetic and geometry of algebraic cycles (Banff, AB, 1998), 47--92, \emph{CRM Proc. Lecture Notes} {\bf 24}, Amer. Math. Soc., Providence, RI, 2000. 


\bibitem{Es1}  Esnault, H\'el\`ene Recent developments on characteristic classes of flat bundles on complex algebraic manifolds. \emph{Jahresber. Deutsch. Math.-Verein.} {\bf 98} (1996), 182--191. 

\bibitem{Es2}  Esnault, H\'el\`ene On flat bundles in characteristic $0$ and $p>0$. \emph{European Congress of Mathematics}, 301--313, Eur. Math. Soc., Zürich, 2013.

\bibitem{EH1} Esnault,   H\'el\`ene; Harris, Michael 
Chern classes of automorphic vector bundles. \emph{Pure Appl. Math. Q.} {\bf 13} (2017),  193--213. 

\bibitem{EH2}  Esnault,   H\'el\`ene; Harris, Michael Chern classes of automorphic vector bundles, II. \emph{\'Epijournal G\'eom. Alg\'ebrique} {\bf 3} (2019), Art. 14, 28 pp. 


\bibitem{EV} Esnault, H\'el\`ene; Viehweg, Eckart Chern classes of Gauss-Manin bundles of weight 1 vanish. 
\emph{K-Theory} {\bf 26} (2002),  287--305.  



\bibitem{HTT} Hotta, Ryoshi; Takeuchi, Kiyoshi; Tanisaki, Toshiyuki D-modules, perverse sheaves, and representation theory. Translated from the 1995 Japanese edition by Takeuchi. \emph{Progress in Mathematics} {\bf 236}. Birkh\"auser Boston, Inc., Boston, MA, 2008. xii+407 pp. 

\bibitem{HL}  Huybrechts, Daniel; Lehn, Manfred The geometry of moduli spaces of sheaves. Second edition. \emph{Cambridge Mathematical Library.} Cambridge University Press, Cambridge, 2010. xviii+325 pp.


\bibitem{Iy}  Iyer, Jaya N. N. Chern invariants of some flat bundles in the arithmetic Deligne cohomology. 
\emph{Math. Z.} {\bf 260} (2008),  61--76. 



\bibitem{IS} Iyer, Jaya N.; Simpson, Carlos T. The Chern character of a parabolic bundle, and a parabolic corollary of Reznikov's theorem.  \emph{Geometry and dynamics of groups and spaces}, 439--485,
\emph{Progr. Math.} {\bf 265}, Birkh\"auser, Basel, 2008. 


\bibitem {Ja1} Jannsen, Uwe
Motivic sheaves and filtrations on Chow groups. Motives (Seattle, WA, 1991), 245--302,
\emph{Proc. Sympos. Pure Math.} {\bf 55}, Part 1, Amer. Math. Soc., Providence, RI, 1994.

\bibitem {Ja2} Jannsen, Uwe  Equivalence relations on algebraic cycles. The arithmetic and geometry of algebraic cycles (Banff, AB, 1998), 225--260, \emph{NATO Sci. Ser. C Math. Phys. Sci.} {\bf 548}, Kluwer Acad. Publ., Dordrecht, 2000. 



\bibitem{Ka0}  Katz, Nicholas M. A conjecture in the arithmetic theory of differential equations. \emph{Bull. Soc. Math. France} {\bf 110} (1982), 203--239.

\bibitem{Ka1} Katz, Nicholas M. On the calculation of some differential Galois groups.
\emph{Invent. Math.} {\bf 87} (1987), 13--61. 

\bibitem{Ka} Katz, Nicholas M. Exponential sums and differential equations. \emph{Annals of Mathematics Studies} {\bf 124}. Princeton University Press, Princeton, NJ, 1990. xii+430 pp.

\bibitem {Kl} Kleiman, Steven L. The Picard scheme. Fundamental algebraic geometry, 235--321,
\emph{Math. Surveys Monogr.} {\bf 123}, Amer. Math. Soc., Providence, RI, 2005. 



\bibitem {La1} Langer, Adrian Semistable sheaves in positive characteristic.
\emph{Ann. of Math.} {\bf 159} (2004), 251--276.




\bibitem {La4} Langer, Adrian Nearby cycles and semipositivity in positive characteristic,  arXiv:1711.05252, 
preprint (2019) 

\bibitem {LS} Langer, Adrian; Simpson, Carlos Rank 3 rigid representations of projective fundamental groups. 
\emph{Compos. Math.} {\bf 154} (2018), 1534--1570.

\bibitem {Laz}  Lazarsfeld, Robert Positivity in algebraic geometry. I. Classical setting: line bundles and linear series. \emph{Ergebnisse der Mathematik und ihrer Grenzgebiete. 3. Folge. A Series of Modern Surveys in Mathematics} {\bf  48}. Springer-Verlag, Berlin, 2004. xviii+387 pp.


\bibitem {Mu} Mumford, David Rational equivalence of $0$-cycles on surfaces.
\emph{J. Math. Kyoto Univ.} {\bf 9} (1968), 195--204. 

 
\bibitem {Re} Reznikov, Alexander All regulators of flat bundles are torsion. \emph{Ann. of Math. (2)} {\bf 141} (1995), 373--386.
 
\bibitem {Re2}  Reznikov, Alexander Analytic topology of groups, actions, strings and varieties. Geometry and dynamics of groups and spaces, 3--93, \emph{Progr. Math.} {\bf 265}, Birkh\"auser, Basel, 2008. 
 


\bibitem {Si0}  Simpson, Carlos T.  Higgs bundles and local systems. \emph{Hautes \'Etudes Sci. Publ. Math.}
 No. {\bf 75} (1992), 5--95.

\bibitem {Si1} Simpson, Carlos T. Some families of local systems over smooth projective varieties.
\emph{Ann. of Math. (2)} {\bf 138} (1993), 337--425. 

\bibitem {Si}  Simpson, Carlos T. Moduli of representations of the fundamental group of a smooth projective variety. II. \emph{Hautes \'Etudes Sci. Publ. Math.} No. {\bf 80} (1994), 5--79 (1995). 


\bibitem {Vo} Voisin, Claire Remarks on filtrations on Chow groups and the Bloch conjecture.
\emph{Ann. Mat. Pura Appl. (4)} {\bf 183} (2004), 421--438. 

\bibitem{Zu1} Zuo, Kang Factorizations of nonrigid Zariski dense representations of $\pi _1$ of projective algebraic manifolds. 
\emph{Invent. Math.} {\bf 118} (1994), 37--46.

\bibitem{Zu2} Zuo, Kang Representations of fundamental groups of algebraic varieties. \emph{Lecture Notes in Mathematics} {\bf 1708}. Springer-Verlag, Berlin, 1999. viii+135 pp.

\end{thebibliography}
\end{document}